\documentclass[preprint,12pt]{elsarticle}
\mathchardef\mhyphen="2D
\usepackage{amssymb}
\usepackage{algcompatible}
\usepackage{algpseudocode}
\usepackage{multirow}
\usepackage{multirow}
\usepackage{xspace}
\usepackage{hyperref}
\usepackage{listings}
\usepackage{url}
\usepackage{csquotes}
\usepackage[mathscr]{euscript} 
\usepackage{amsfonts,epsfig}
\usepackage{enumerate}
\usepackage{amsmath,amscd,amsfonts,amssymb}

\usepackage{lineno,hyperref}
\usepackage{graphicx,subfigure}
\usepackage{amsfonts,epsfig}

\usepackage[mathscr]{euscript} 
\usepackage{enumerate}
\usepackage{amsmath,amscd,amsfonts,amssymb}
\usepackage{times}
\usepackage{extarrows}
\usepackage{stmaryrd}
\usepackage{textcomp}
\usepackage[mathscr]{euscript}
\usepackage{algorithm2e}
\usepackage{colortbl}
\usepackage{color,soul}
\usepackage{newunicodechar}
\usepackage{amsthm}
\theoremstyle{definition}
\newtheorem{thm}{Theorem}

\newtheorem{lem}[thm]{Lemma}
\newtheorem{defn}{Definition}
\newtheorem{exa}{Example}

\newcommand{\R}{{\bf R}}

\begin{document}

\begin{frontmatter}

\title{Randomized block Krylov method for approximation of truncated tensor SVD}

\author[label1]{Malihe Nobakht Kooshkghazi$^\dagger$}
\affiliation[label1]{organization={Independent researcher, m.nobakht88@gmail.com},
            }

\author[inst1]{Salman Ahmadi-Asl$^{\dagger,*}$}

\affiliation[inst1]{organization={Research Center of the Artificial Intelligence Institute,  and Lab of Machine Learning and Knowledge Representation, Innopolis University, 420500},
            city={Innopolis},
            country={Russia, {\em s.ahmadiasl@innopolis.ru}}}

\author[inst2]{André L. F. de Almeida}

\affiliation[inst2]{organization={Department of Teleinformatics Engineering, Federal University of Ceara, Fortaleza, Brazil,\\ $^\dagger$Equal contribution,\,\,$^*$Corresponding author},
           }

\begin{abstract}
This paper focuses on exploring the use of the block Krylov subspace method for approximating the truncated tensor singular value decomposition (T-SVD). We present the theoretical findings of our proposed randomized technique. To validate the effectiveness and practicality of this approach, we conduct several numerical experiments using both synthetic and real-world datasets. The results indicate that our method yields promising outcomes. Additionally, we demonstrate applications of this approach in data completion and data compression.
\end{abstract}


\end{frontmatter}


\section{Introduction}
\label{sec:intro}
Tensors are multidimensional arrays that generalize scalars, vectors, and matrices to higher dimensions. While a scalar is a zero-dimensional tensor, a vector is a one-dimensional tensor, and a matrix is a two-dimensional tensor, tensors can have any number of dimensions. Tensors are fundamental in neural networks, where data inputs, weights, and outputs are often represented as tensors. Tensors are widely used in various fields, including signal processing \cite{Sidiropoulos2017,Miron2020,Yassin_Andre}, telecommunications \cite{Almeida_Elsevier_2007, Ximenes2015TSP,dearaujo2021channel}, machine learning \cite{cichocki2017tensor}, physics \cite{orus2019tensor}, Hammerstien system identification \cite{elden2019solving} and data analysis \cite{ahmadi2023fast,asante2021matrix,ahmadi2023efficient}, due to their ability to capture complex, high-dimensional relationships in data. They provide a natural way to represent high-dimensional data, such as color images (3rd-order tensors: height $\times$ width $\times$ color channels), video sequences (4th-order tensors: height $\times$ width $\times$ color channels $\times$ time), hyperspectral images (3rd-order tensors: height $\times$ width $\times$ spectral bands), social network data (3rd-order tensors: users $\times$ users $\times$ interactions), for more applications of tensors, see \cite{comon2009tensor,kolda2009tensor} and the references therein. 

Similar to SVD, which is used to analyze matrices, tensor decompositions are mathematical tools that generalize matrix decompositions (e.g., SVD, PCA) to higher-order tensors. They aim to represent a tensor as a combination of simpler, lower-dimensional components. The most common tensor decompositions include: Canonical polyadic decomposition \cite{hitchcock1927expression}, Tucker decomposition \cite{tucker1964extension}, tensor train decomposition \cite{oseledets2011tensor}, tensor ring decomposition \cite{zhao2016tensor} and tensor SVD (T-SVD) \cite{kilmer2011factorization}. Tensor decompositions have numerous applications across various domains, such as signal processing to machine learning, tensors are indispensable tools in modern data science. The randomized block Krylov subspace method was proposed in \cite{musco2015randomized} and was shown to be a very efficient technique for the low-rank approximation of matrices. Motivated by these results, it was extended to several tensor decompositions. For example, it was previously utilized in \cite{qiu2024towards} and \cite{yu2023randomized} for Tucker and TT decompositions, respectively. It is known that the randomized block Krylov subspace method can better capture the range of a matrix compared to the classical randomized algorithms, and this leads to obtaining  better low-rank approximations. The integration of randomization and Krylov methods is a powerful paradigm shift. Randomization provides a cheap, parallelizable way to guess a good subspace, while Krylov methods provide a powerful, iterative way to refine that guess to high accuracy. The synergy between them creates algorithms that are faster, more robust, and come with stronger theoretical guarantees than either approach used in isolation. This has made them the method of choice for large-scale low-rank matrix approximations in both academic research and industrial practice. It is worth highlighting that  randomized algorithms are applicable for computing low-rank approximations of data tensors that have low-rank structures, such as images and videos.

In this paper, we propose to use the randomized block Krylov method for the computation of the truncated Tensor SVD (T-SVD). It is another extension of the method proposed in \cite{musco2015randomized} to tensors. To the best of our knowledge, this is the first paper discussing the possibility of using the randomized block Krylov method for the T-SVD model. A detailed theoretical analysis of the proposed algorithm is discussed, and they are supported by simulation results. An application of the proposed randomized method to image completion is given. Our simulation results show the effectiveness of the proposed randomized algorithm for the synthetic and real-world datasets.

This work makes the following contributions:
\begin{itemize}
\item We introduce a randomized block Krylov method to efficiently approximate the truncated T-SVD.
\item We present a detailed theoretical analysis, establishing guarantees for the accuracy of the proposed algorithm.
\item Through extensive simulations, we validate our theoretical findings and demonstrate the practical utility of the method on image completion problems.
\end{itemize}

The paper is organized as follows: Section \ref{sec:intro}, is devoted to presenting the preliminary concepts and definitions. The proposed randomized algorithm is outlined in Section \ref{Sec:prop} and in Section \ref{Sec:theo} the theoretical results are outlined. The numerical experiments are given in Section \ref{sec:experi}. Finally, the conclusion is presented in Section \ref{sec:conclu}.

\section{Preliminary concepts}\label{sec:intro}
This section presents some definitions and lemmas that will be used in the sequel. We denote tensors with calligraphic letters and matrices with capital letters. The set of real and natural numbers denote by $\mathbb{R}$ and $\Bbb N$, respectively. Let $O$, and $\mathcal{O}$ stand for a zero matrix, and a zero tensor, respectively. The identity matrix of size $n\times n$ is denoted by $ I_{n}$.

Let $\mathcal{A} \in \mathbb{R}^{n_1 \times n_2 \times n_3}$ be a third-order tensor. We denote its elements as $\mathcal{A}(i,j,k)$, where $i = 1, \dots, n_1$, $j = 1, \dots, n_2$, and $k = 1, \dots, n_3$. A third-order tensor can be viewed as a stack of frontal slices. The $k$-th frontal slice of $\mathcal{A}$ is denoted as ${A}^{(k)} \in \mathbb{R}^{n_1 \times n_2}$. Also, $ \mathcal{A}(:,j,:) $ gives us the $j$-th lateral slice of tensor $\mathcal{A}$. A tube of $\mathcal{A}$ is a vector obtained by fixing the first two indices and varying the third, e.g., $\mathcal{A}(i,j,:)$. Let $\mathcal{X}, \mathcal{Y} \in \Bbb R^{n_{1}\times n_{2} \times n_{3}}$. The inner product between $\mathcal{X}$ and $\mathcal{Y}$ is defined as
$
 <\mathcal{X} , \mathcal{Y} >= \sum\limits_{i,j,k} x_{ijk}y_{ijk}.$ The notation ``conj'' is used to denote the component-wise conjugate of a complex tensor.

\begin{defn}[\cite{kilmer2011factorization}]
The discrete Fourier transform (DFT) matrix $ F_{n} $ is of the form
\begin{equation*}
F_{n}=\begin{bmatrix}
1 & 1 & 1& \ldots &1 \\
1 & \omega & \omega ^{2}& \ldots & \omega^{(n-1)} \\
\vdots & \vdots & \vdots &\ddots & \vdots\\
1& \omega^{n-1} &\omega^{2(n-1})& \ldots & \omega^{(n-1)(n-1)}\\
\end{bmatrix}\in \mathbb{C}^{n\times n},
\end{equation*}
where $\omega=e^{-\frac{2\pi i}{n}}$. 
Note that $\dfrac{F_{n}}{\sqrt{n}}$ is an orthogonal matrix, in the sense that 
\[
F_{n}^{\ast} F_{n}=F_{n} F_{n}^{\ast}=nI_{n}.
\]
Thus, $ F_{n}^{-1}=\dfrac{F_{n}^{\ast}}{n} $. Now, we consider the DFT on tensors. For $\mathcal{X} \in \mathbb{R}^{n_{1}\times n_{2} \times n_{3}}$, we denote by $\overline{\mathcal{X}} \in \mathbb{C}^{n_{1}\times n_{2} \times n_{3}}$ the result of the DFT on $\mathcal{X}$ along the third dimension, that is, performing the DFT on all the tubes of $\mathcal{X}$. By using the MATLAB command $\mathtt{fft}$, we obtain
\[
 \overline{\mathcal{X}} = \mathtt{fft}(\mathcal{X}, [ ], 3).
\]
Similarly, we can compute $ \mathcal{X} $ in terms of $  \overline{\mathcal{X}} $ using the inverse $\mathtt{fft}$:
\[
\mathcal{X} = \mathtt{ifft}(\overline{\mathcal{X}}, [ ], 3).
\]
We denote by $ \overline{X} \in \mathbb{C}^{n_{1}n_{3}\times n_{2}n_{3}}$ a block diagonal matrix with its $i\mathrm{th}$ block on the diagonal as the $i\mathrm{th}$ frontal slice $ \overline{X}^{(i)}$ of $\overline{\mathcal{X}}$, that is,
\begin{equation*}
\overline{ X} = \mathtt{bdiag}(\overline{\mathcal{X}})=\begin{bmatrix}
\overline{X}^{(1)}& \\
&\overline{X}^{(2)}& \\
 &  &\ddots & \\
&  &&  &\overline{X}^{(n_{3})}\\
\end{bmatrix},
\end{equation*}
where $\mathtt{bdiag}$ is an operator that maps the tensor $ \overline{\mathcal{X}} $ to the block diagonal matrix $ \overline{X}$. Also, we define the block circulant matrix $\mathtt{bcirc}(\mathcal{X}) \in \mathbb{R}^{n_{1}n_{3}\times n_{2}n_{3}}$ of $ \mathcal{X} $ by
\begin{equation*}
\mathtt{bcirc}(\mathcal{X})=\begin{bmatrix}
{ X}^{(1)}& { X}^{(n_{3})}&\ldots & {X}^{(2)}\\
{X}^{(2)}&{X}^{(1)}& \ldots & {X}^{(3)} \\
\vdots &\vdots  &\ddots &\vdots \\
{X}^{(n_{3})}& {X}^{(n_{3}-1)}  & \ldots &{X}^{(1)}\\
\end{bmatrix}.
\end{equation*}
\end{defn}
From \cite{kilmer2013third} the following relation shows that  the block circulant matrix can be block diagonalized, i.e
\begin{eqnarray}\label{mn35}
\overline{X}=(F_{n_{3}}\otimes I_{n_{1}}) \mathtt{bcirc}(\mathcal{X})(F^{*}_{n_{3}}\otimes I_{n_{2}}),
\end{eqnarray}
where $\mathcal{X}$ is of dimension $n_1 \times n_2 \times n_3$.
\begin{defn}[\cite{kilmer2011factorization}]
Let $ \mathcal{X} $ be a tensor of order three and dimension $n_{1}\times n_{2}\times n_{3}$. Then, the operations unfold and fold are defined by
\begin{equation*}
\mathtt{unfold}(\mathcal{X})=\begin{bmatrix}
{X}^{(1)}\\
{X}^{(2)}\\
\vdots \\
{X}^{(n_{3})}
\end{bmatrix},\ \mathtt{fold}(\mathtt{unfold}(\mathcal{X}))=\mathcal{X},
\end{equation*}
where $\mathtt{unfold}$ maps $ \mathcal{A} $ to a matrix of size $n_{1}n_{3} \times n_{2}$ and $\mathtt{fold}$ is its inverse operator.
\end{defn}
 Kilmer and her collaborators \cite{kilmer2011factorization} introduced the T-product operation between two third-order tensors as a useful generalization of matrix multiplication for tensors of order three. By using this multiplication, tensor factorizations such as SVD, QR, and eigen decompositions were deﬁned on tensors similar to matrices \cite{kilmer2013third}. The T-product has been applied in signal processing \cite{semerci2014tensor}, machine learning \cite{settles2007multiple}, computer vision \cite{zhang2014novel}. We now recall this definition.
\begin{defn} \label{m11}
Let $\mathcal{X} \in \mathbb{R}^{n_{1}\times n_{2}\times n_{3}}$ and $ \mathcal{Y} \in \mathbb{R}^{n_{2}\times n_4 \times n_{3}}$. Then, T-product, $ \mathcal{X}* \mathcal{Y} $ is defined to be the tensor 
\[
\mathcal{X} * \mathcal{Y}=\mathtt{fold}(\mathtt{bcirc}(\mathcal{X}).\mathtt{unfold}(\mathcal{Y})),
\]
of dimension $ n_{1}\times n_4 \times n_{3} $. 
The $T$-product is equivalent to matrix multiplication in the Fourier domain; that is, $\mathcal{W} = \mathcal{X} * \mathcal{Y}$ is equivalent to $\overline{W} =\overline{ X}\,\overline{Y}$. Specifically, the T-product can be computed in the Fourier domain using Algorithm  \ref{ALG:TP}.

For a tensor $\mathcal{A}\in\mathbb{R}^{n_1\times n_2\times n_3}$, the range is defined as follows \cite{kilmer2013third}
\[
{\rm Range}({\mathcal A})=\{\mathcal{A}*\mathcal{X},\,\,\,\forall \mathcal{X}\in\mathbb{R}^{n_2\times n_4\times n_3}\}.
\]  
\end{defn} 
\begin{defn}[\cite{Ichi},\cite{nobakht2023}]
Let $\mathcal{A} \in \Bbb R^{n_{1}\times m_{1}\times n_{3}}, \mathcal{B} \in \Bbb R^{n_{1}\times m_{2}\times n_{3}}, \mathcal{C} \in \Bbb R^{n_{2}\times m_{1}\times n_{3}} $, and $\mathcal{D} \in \Bbb R^{n_{2}\times m_{2}\times n_{3}}$ be tensors. The block tensor 
\[
 \left[\begin{array}{cccc}
\mathcal{A} & \mathcal{B} \\
 \mathcal{C} & \mathcal{D} 
\end{array} \right] \in \Bbb R^{(n_{1}+n_{2})\times (m_{1}+m_{2})\times n_{3}},
\]
is defined by compositing the frontal slices of the four tensors.
\end{defn}
\begin{lem}[\cite{Ichi},\cite{nobakht2023}]
Let $\mathcal{A}, \mathcal{A}_{1} \in \Bbb R^{n\times s \times n_{3}}, \mathcal{B}, \mathcal{B}_{1} \in \Bbb R^{n\times p \times n_{3}}, \mathcal{A}_{2} \in \Bbb R^{l \times s \times n_{3}}, \mathcal{B}_{2} \in \Bbb R^{l\times p \times n_{3}}, \mathcal{C} \in \Bbb R^{s \times n \times n_{3}}, \mathcal{D} \in \Bbb R^{p\times n \times n_{3}}$, and $ \mathcal{F} \in \Bbb R^{n\times n \times n_{3}}$. Then
\begin{itemize}
\item[1-]
$
\mathcal{F} * [\mathcal{A}\quad \mathcal{B}]=[ \mathcal{F} * \mathcal{A} \quad \mathcal{F} * \mathcal{B}]\in \Bbb R^{n \times (s+p)\times n_{3}}$;
\item[2-]  
$
\left[\begin{array}{cccc}
\mathcal{C} \\
 \mathcal{D}  
\end{array} \right]* \mathcal{F}=\left[\begin{array}{cccc}
\mathcal{C}* \mathcal{F} \\
 \mathcal{D} * \mathcal{F} 
\end{array} \right] \in \Bbb R^{(s+p)\times n \times n_{3}}
$;
\item[3-] 
$
[\mathcal{A} \quad \mathcal{B}] * \left[\begin{array}{cccc}
\mathcal{C} \\
 \mathcal{D}  
\end{array} \right]=\mathcal{A}* \mathcal{C}+ \mathcal{B} * \mathcal{D} \in \Bbb R^{n\times n \times n_{3}}
$;
\item[4-] 
$
\left[\begin{array}{cccc}
\mathcal{A}_{1}&\mathcal{B}_{1} \\
 \mathcal{A}_{2}& \mathcal{B}_{2}
\end{array} \right] * \left[\begin{array}{cccc}
\mathcal{C} \\
 \mathcal{D}
\end{array} \right]= \left[\begin{array}{cccc}
\mathcal{A}_{1} * \mathcal{C}+ \mathcal{B}_{1} *\mathcal{D}\\
 \mathcal{A}_{2}* \mathcal{C}+ \mathcal{B}_{2} * \mathcal{D}
\end{array} \right] \in \Bbb R^{(l+n) \times n \times n_{3}}
$.
\end{itemize}
\end{lem}
\begin{defn}\cite{cao2023some}\label{mn3}
The Frobenius norm and the spectral norm of $ \mathcal{X} \in \Bbb R^{n_{1}\times n_{2} \times n_{3}}$ are defined in the following form
\[
\parallel \mathcal{X} \parallel _{F}^{2}=\sum\limits_{i,j,k} \vert x_{ijk}\vert^{2},\quad \parallel \mathcal{X} \parallel _{2}=\parallel \mathtt{bcirc}( \mathcal{X})\parallel_{2}.
\]

\end{defn}
It can be proved that for a tensor $\mathcal{X}   \in\mathbb{R}^{n_1\times n_2\times n_3}$, we have 
\begin{eqnarray}\label{eq_fou}
\|{\mathcal X}\|^2_F=\frac{1}{n_3}\sum_{i=1}^{n_3}\|\overline{{\mathcal X}}(:,:,i)\|_F^2, \quad (\parallel \mathcal{X}\parallel_{F}=\frac{1}{\sqrt{n_{3}}}\parallel \overline{X}\parallel_{F}).
\end{eqnarray}
where $\overline{\mathcal{X}}(:,:,i)$ is the $i$-th frontal slice of the tensor $\overline{\mathcal{X}}=\mathtt{fft}(\mathcal{X},[],3)$, which calculates the fast Fourier transform of all tubes of the tensor $X$, see \cite{zhang2018randomized}.
\begin{lem}\label{mn43}
Let $ \mathcal{X} \in \Bbb R^{n_{1}\times n_{2}\times n_{3}} $ and $\mathcal{Y} \in \Bbb R^{n_{2}\times n_{4}\times n_{3}}$. Then,
\[\parallel \mathcal{X} * \mathcal{Y}\parallel_{2}\leq\parallel \mathcal{X}\parallel_{2}
\parallel \mathcal{Y}\parallel_{2}.\]
\begin{proof} The following computations 
\begin{eqnarray*}
\parallel \mathcal{X} *\mathcal{Y}\parallel_{2}= \parallel\mathtt{bcirc}(\mathcal{X})\mathtt{bcirc}(\mathcal{Y})\parallel_2\leq \parallel\mathtt{bcirc}(\mathcal{X})\parallel_2  \parallel\mathtt{bcirc}(\mathcal{Y})\parallel_2 . 
\end{eqnarray*}
completes the proof.
\end{proof}
\end{lem}
\begin{defn}\cite{kilmer2011factorization,zheng2021t}
The transpose of a tensor $\mathcal{X} \in \mathbb{R}^{n_1 \times n_2 \times n_3}$, denoted $\mathcal{X}^\top \in \Bbb R^{n_2 \times n_1 \times n_3}$, is obtained by transposing each frontal slice and then reversing the order of the transposed slices from $2$ to $n_3$. Mathematically,
\[
\mathcal{X}^\top(i,j,1) = \mathcal{X}(j,i,1),\quad\mathcal{X}^\top(i,j,k) = \mathcal{X}(j,i,n_3 +2 -k),\,\,k=2,\ldots,I_3.
\] 
Also, suppose that $\mathcal{X}$ and $\mathcal{Y}$ are tensors such that $ \mathcal{X}\star \mathcal{Y} $ and $ \mathcal{Y}^{T}\star  \mathcal{X}^{T} $ are defined\footnote{The dimensional requirements for the t-product ($\mathcal{X} \in \mathbb{R}^{n_1 \times n_2 \times n_3}$ and $\mathcal{Y} \in \mathbb{R}^{n_2 \times n_4 \times n_3}$) and for tensor addition (both tensors must have identical dimensions $n_1 \times n_2 \times n_3$).}. Then
\[
(\mathcal{X} \star \mathcal{Y})^{T}=\mathcal{Y}^{T}\star \mathcal{X}^{T}, \quad (\mathcal{X}+ \mathcal{Y})^{\top}=\mathcal{X}^{\top}+\mathcal{Y}^{\top} .
\]
\end{defn}
A tensor is called T-symmetric if $\mathcal{X}^T=\mathcal{X}$. 
\\
\begin{lem}\label{mn2}\cite{zheng2021t}, \cite{lund2020tensor}
Let $\mathcal{X} \in \mathbb{R}^{n_{1}\times n_{2}\times n_{3}}$ and let $ \mathcal{Y} \in \mathbb{R}^{n_{2}\times n_4 \times n_{3}}$. Then 
\begin{itemize}
    \item 
    The matrix $\mathtt{bcirc}$ associated to $ \mathcal{X}*\mathcal{Y}$ is related to $ \mathtt{bcirc}(\mathcal{X})$ and $ \mathtt{bcirc}(\mathcal{Y})$,
\[
\mathtt{bcirc}(\mathcal{X}*\mathcal{Y})=\mathtt{bcirc}(\mathcal{X}) \mathtt{bcirc}(\mathcal{Y}).
\]
\item  
$\mathtt{bcirc}(\mathcal{X}^{\top})=(\mathtt{bcirc}(\mathcal{X}))^T.$
\item 
For $k=0,1,2,3,\dots$, 
\[
\mathtt{bcirc}(\mathcal{X}^{k})= (\mathtt{bcirc}(\mathcal{X}))^{k}.
\]
\item 
 The operator $\mathtt{bcirc}$ is linear, that is, 
 \[
 \mathtt{bcirc}(\alpha \mathcal{X}+ \beta \mathcal{Y})=\alpha \mathtt{bcirc}(\mathcal{X})+ \beta \mathtt{bcirc}(\mathcal{Y}),
 \]
 where  $\mathcal{X}$ and $\mathcal{Y}$ are tensors of the same dimension, and $\alpha, \beta \in \Bbb C$. 
\end{itemize}
\end{lem}
\begin{defn}\cite{cao2023some}
The identity tensor $\mathcal{I} \in \mathbb{R}^{n \times n \times n_3}$ is a tensor whose first frontal slice is the $n \times n$ identity matrix, and all other frontal slices are zero matrices.   
\end{defn}
\begin{defn}\cite{cao2023some}
A tensor $\mathcal{Q} \in \mathbb{R}^{n \times n \times n_3}$ is orthogonal if:
\[
\mathcal{Q}^\top * \mathcal{Q} = \mathcal{Q} * \mathcal{Q}^\top = \mathcal{I}.
\]
\end{defn}
From Lemma \ref{mn2}, it is obvious that the matrix $\mathtt{bcirc}$ associated with an orthogonal tensor is an orthogonal matrix. 

\begin{defn}\cite{cao2023some}, \cite{ju2024geometric}
The T-symmetric tensor $ \mathcal{A} \in \Bbb R^{n \times n \times n_{3}}$ is called  T-positive semideﬁnite tensor if 
\[
\big< \mathcal{X}, \mathcal{A}*\mathcal{X} \big> \geq 0, \quad \forall \mathcal{X} \in \Bbb R^{n \times 1 \times n_{3}}, \quad \mathcal{X}\neq \mathcal{O}.
\]
It can be shown that $\mathcal{A}*\mathcal{A}^{\top}$ is a T-positive semideﬁnite tensor \cite{ju2024geometric}.
\end{defn}
\begin{lem}\cite{ju2024geometric}\label{mn8}
Let $ \mathcal{X} \in \Bbb R^{n \times n \times n_{3}} $. A T-symmetric tensor $\mathcal{X}$ is T-positive semidefinite if and only if $ \mathtt{bcirc}(\mathcal{X}) $ is a symmetric positive semidefinite matrix.
\end{lem}
The analysis of T-positive semideﬁnite tensor demands the introduction of eigenvalues of a tensor under T-product.
\begin{defn}\cite{zheng2021t}\label{mn45}
Let $ \mathcal{A} \in \Bbb C^{n \times n \times n_{3}}$, $ \mathcal{X} \in \Bbb C^{n \times 1 \times n_{3}}$, and $\mathcal{X}\neq \mathcal{O}$. If $\lambda$ and $ \mathcal{X}$ satisfies 
\[
\mathcal{A} * \mathcal{X}=\lambda \mathcal{X},
\]
then $\lambda$ is called a T-eigenvalue of $\mathcal{A}$ and $\mathcal{X}$ is a T-eigenvector of $\mathcal{A}$ associated to $\lambda$. Clearly, for each T-eigenvalue $\lambda$ of $\mathcal{A} $,
\[
\mathtt{bcirc}(\mathcal{A}) \mathtt{unfold}(\mathcal{X})= \lambda \, \mathtt{unfold}(\mathcal{X}),
\]
and this shows that T-eigenvalues of $\mathcal{A} $ are the eigenvalues of the matrix $\mathtt{bcirc}(\mathcal{A})$.
\end{defn}
\begin{lem}\label{mn77}\cite{zheng2021t}
The T-symmetric tensor $ \mathcal{X} \in \Bbb R^{n \times n\times n_{3}} $ is T-positive semidefinite (positive definite) if and only if the T-eigenvalues of $\mathcal{X}$ are non-negative (positive).
\end{lem}

The T-SVD is a generalization of the matrix Singular Value Decomposition (SVD) to third-order tensors. It is based on the t-product, which is a specific type of tensor-tensor product that operates on third-order tensors. 
Given a tensor $\mathcal{X} \in \mathbb{R}^{n_1 \times n_2 \times n_3}$, its T-SVD is given by:
\[
\mathcal{X} = \mathcal{U} * \mathcal{S} * \mathcal{V}^\top,
\]
where $\mathcal{U} \in \mathbb{R}^{n_1 \times n_1 \times n_3}$ is an orthogonal tensor, $\mathcal{S} \in \mathbb{R}^{n_1 \times n_2 \times n_3}$ is an f-diagonal tensor (i.e., each frontal slice is a diagonal matrix), $\mathcal{V} \in \mathbb{R}^{n_2 \times n_2 \times n_3}$ is an orthogonal tensor. The number of nonzero tubes in the middle tensor $\mathcal{S}$ is called the tubal rank of $\mathcal{X}$ and is denoted by $T-rank$. The nonzero tubes of the tensor $\mathcal{S}$, are called singular tubes. The truncated T-SVD of tubal rank $R$ for the tensor ${\mathcal X}$ is 
\begin{eqnarray*}
{\mathcal X}\approx {\mathcal U}_R * {\mathcal S}_R* {\mathcal V}_R^\top,    
\end{eqnarray*}
where ${\mathcal U}_R={\mathcal U}(:,1:R,:)\in\mathbb{R}^{n_1\times R\times n_3},\,{\mathcal V}_R={\mathcal V}(:,1:R,:)\in\mathbb{R}^{n_2\times R\times n_3},$ and ${\mathcal S}_R={\mathcal S}(1:R,1:R,:)\in\mathbb{R}^{R \times R\times n_3}$. The T-SVD can be computed by performing the matrix SVD on each frontal slice of $\mathcal{A}$ in the Fourier domain and then transforming back to the spatial domain. Similarly, the tensor QR (T-QR) of a tensor $\mathcal{X}\in\mathbb{R}^{n_1\times n_2\times n_3}$ can be defined similarly as the following representation:
\begin{eqnarray*}
\mathcal{X}=\mathcal{Q}*\mathcal{R},    
\end{eqnarray*}
where $\mathcal{Q}\in\mathbb{R}^{n_1\times n_1\times n_3}$ is an orthogonal tensor and the tensor $\mathcal{R}\in\mathbb{R}^{n_1\times n_2\times n_3}$  is an f-upper triangular tensor (each frontal slice is an upper triangular matrix). The process of T-QR decomposition of a tensor is presented in Algorithm \ref{ALG:TQR}.
\begin{defn}\cite{jin2023moore}
Let $\mathcal{A} \in \Bbb R^{n_{1}\times n_{2} \times n_{3}}$. If there exists a tensor $ \mathcal{X} \in \Bbb R^{n_{2}\times n_{1} \times n_{3}} $ such that
\begin{eqnarray*}
& &\mathcal{A}*\mathcal{X}*\mathcal{A}=\mathcal{A}, \quad \mathcal{X}*\mathcal{A}*\mathcal{X}=\mathcal{X}\\
& &(\mathcal{A}*\mathcal{X})^{\top}=\mathcal{A}*\mathcal{X}, \quad (\mathcal{X}*\mathcal{A})^{\top}=\mathcal{X}*\mathcal{A},
\end{eqnarray*}
then $ \mathcal{X} $ is called the Moore–Penrose (MP) inverse of the tensor $ \mathcal{A} $ and is denoted by $\mathcal{A}^{\dagger}$.
\end{defn}
The MP inverse of a tensor can be computed in the Fourier domain and summarized in Algorithm \ref{ALG:mp}. Similar to the matrix case, there is an explicit formula for the MP of a tensor. More precisely, for a given tensor $\mathcal{A}\in\mathbb{R}^{n_1\times n_2\times n_3}$, and its T-SVD as $\mathcal{A}=\mathcal{U}*\mathcal{S}*\mathcal{V}^T$, we have $\mathcal{A}^\dagger=\mathcal{V}*\mathcal{S}^\dagger*\mathcal{U}^\top$, where $\mathcal{S}^\dagger$ is obtained by transposing the tensor $\mathcal{S}$ and replacing the nonzero tubes with their inverses. From these facts, it is now easy to see that
\begin{eqnarray}\label{mn41}
\mathcal{A}*\mathcal{A}^\dagger=\mathcal{U}*\mathcal{U}^\top, \quad
\mathcal{A}^\dagger*\mathcal{A}=\mathcal{V}*\mathcal{V}^\top.
\end{eqnarray}
\RestyleAlgo{ruled}
\LinesNumbered
\begin{algorithm}
{\small
\SetKwInOut{Input}{Input}
\SetKwInOut{Output}{Output}\Input{Two data tensors $\mathcal{X} \in {\mathbb{R}^{{n_1} \times {n_2} \times {n_3}}},\,\,{\mathcal Y} \in {\mathbb{R}^{{n_2} \times {n_4} \times {n_3}}}$} 
\Output{T-product ${\mathcal C} = {\mathcal X} * {\mathcal Y}\in\mathcal{R}^{n_1\times n_4\times n_3}$}
\caption{T-product in the Fourier domain \cite{kilmer2011factorization}}\label{ALG:TP}
      {
      $\overline{{\mathcal X}} = \mathtt{fft}\left( {{\mathcal X},[],3} \right)$;\\
      $\overline{{\mathcal Y}} = \mathtt{fft}\left( {{\mathcal Y},[],3} \right)$;\\
\For{$i=1,2,\ldots,\lceil \frac{n_3+1}{2}\rceil$}
{                        
$\overline{{\mathcal C}}\left( {:,:,i} \right) = \overline{{\mathcal X}}\left( {:,:,i} \right)\,\overline{{\mathcal Y}}\left( {:,:,i} \right)$;\\
}
\For{$i=\lceil\frac{n_3+1}{2}\rceil+1\ldots,n_3$}{
$\overline{{\mathcal C}}\left( {:,:,i} \right)={\rm conj}(\overline{{\mathcal C}}\left( {:,:,n_3-i+2} \right))$;
}
${\mathcal  C} = \mathtt{ifft}\left( {\overline{{\mathcal  C}},[],3} \right)$;   
       	} 
        }
\end{algorithm}

\RestyleAlgo{ruled}
\LinesNumbered
\begin{algorithm}
\SetKwInOut{Input}{Input}
\SetKwInOut{Output}{Output}\Input{The data tensor ${\mathcal X} \in {\mathbb{R}^{{n_1} \times {n_2} \times {n_3}}}$} 
\Output{The T-QR computation $\mathcal X=\mathcal Q*\mathcal R$.}
\caption{Fast T-QR decomposition of the tensor ${\mathcal  X}$}\label{ALG:TQR}
      {
      $\overline{{\mathcal X}} = \mathtt{fft}\left( {{\mathcal X},[],3} \right)$;\\
\For{$i=1,2,\ldots,\lceil \frac{n_3+1}{2}\rceil$}
{                        
$[\overline{{\mathcal Q}}\left( {:,:,i} \right),\overline{\R}(:,:,i)] = {\rm qr}\,(\overline{\mathcal X}(:,:,i),0)$;\\
}
\For{$i=\lceil\frac{n_3+1}{2}\rceil+1\ldots,n_3$}{
$\overline{{\mathcal Q}}\left( {:,:,i} \right)={\rm conj}(\overline{{\mathcal Q}}\left( {:,:,n_3-i+2} \right))$;\\
$\overline{{\mathcal R}}\left( {:,:,i} \right)={\rm conj}(\overline{{\mathcal R}}\left( {:,:,n_3-i+2} \right))$;
}
${\mathcal Q}= \mathtt{ifft}\left( {\overline{{\mathcal Q}},[],3} \right)$;\\
${\mathcal R}= \mathtt{ifft}\left( {\overline{{\mathcal R}},[],3} \right)$; 
       	}       	
\end{algorithm}

\RestyleAlgo{ruled}
\LinesNumbered
\begin{algorithm}
\SetKwInOut{Input}{Input}
\SetKwInOut{Output}{Output}\Input{The data tensor ${\mathcal X} \in {\mathbb{R}^{{n_1} \times {n_2} \times {n_3}}}$} 
\Output{Moore-Penrose pseudoinverse ${\mathcal X}^{\dagger}\in\mathbb{R}^{n_2\times n_1\times n_3}$}
\caption{Fast Moore-Penrose pseudoinverse computation of the tensor ${\mathcal X}$}\label{ALG:mp}
      {
      $\overline{{\mathcal X}} = \mathtt{ fft}\left( {\mathcal X},[],3 \right)$;\\
\For{$i=1,2,\ldots,\lceil \frac{n_3+1}{2}\rceil$}
{                        
$\overline{{\mathcal C}}\left( {:,:,i} \right) = \mathtt{pinv}\,(\overline{\mathcal X}(:,:,i))$;\\
}
\For{$i=\lceil\frac{n_3+1}{2}\rceil+1\ldots,n_3$}{
$\overline{{\mathcal C}}\left( {:,:,i} \right)={\rm conj}(\overline{{\mathcal C}}\left( {:,:,n_3-i+2} \right))$;
}
${\mathcal X}^{\dagger} = \mathtt{ifft}\left( {\overline{{\mathcal C}},[],3} \right)$;   
       	}       	
\end{algorithm}

\section{Proposed randomized block Krylov method for approximate truncated T-SVD}\label{Sec:prop}
The randomized block Krylov subspace algorithm is an efficient method for approximating large-scale matrix computations, including eigenvalue problems and low-rank matrix approximations. It combines randomization techniques with the power of Krylov subspace methods to achieve high accuracy with reduced computational cost. In the case of matrices as second-order tensors, it is known that the convergence rate of pure Krylov methods depends on the spectral gap. Slow decay (small gap) leads to slow convergence. Randomized sketching with power iterations creates a matrix \((XX^T)^q X\) whose singular values decay faster, effectively enhancing the gap.
Randomized methods provide strong, non-asymptotic probabilistic error bounds. For a target rank \(k\), the approximation error is guaranteed (with high probability) to be within a small factor of the optimal error. Krylov refinement tightens these bounds.

Although a classical randomized algorithm applies power iterations $q\geq 1$ to enhance the decay of singular tubes, it does not use the results from each step of the iteration. Since it only uses the final result, $(\mathcal{X}*\mathcal{X}^\top)^q*\mathcal{X}*\mathcal{B}$, the method generally cannot closely approximate the largest $k$ singular tubes. However, the block randomized algorithm captures a much richer spectral range, leading to faster, super-linear convergence and higher accuracy for the same number of tensor operations. Let us discuss this precisely. The classical randomized algorithm with the power iteration technique for truncated T-SVD is presented in Algorithm \ref{ALg_1}. As can be seen, it uses the last iteration $(\mathcal{X}*\mathcal{X}^\top)^q*\mathcal{X}*\mathcal{B}$ of the power iteration loop for computing an orthonormal basis of the tensor range in Line 7. The proposed randomized algorithm looks for a better matrix to capture the range of the matrix and do this, it uses all iterations $(\mathcal{X}*\mathcal{X}^\top)^i*\mathcal{X}*\mathcal{B},\,\,i=1,2,\ldots,q$, which leads to better capturing the range of the tensor and better approximation results. This process is summarized in Algorithm \ref{ALg_2}. The main difference between the proposed randomized algorithm and the one proposed in \cite{zhang2018randomized} lies in how they compute an orthonormal basis for the range of a tensor to be used in the projection stage within the randomization framework. We extensively examine this idea with detailed theoretical analyses. 

It is worth noting that the primary trade-off is between \textbf{accuracy} and \textbf{computational cost per iteration}:
\begin{itemize}
    \item \textbf{Classical Randomized (with Power Iteration):} Lower cost per iteration. However, it requires more iterations (higher q) to achieve high accuracy, especially for complex spectra, resulting in a high total cost for accurate approximations.
    \item \textbf{Proposed Block Krylov:} Higher cost per iteration due to the orthogonalization of a larger block Krylov basis. The significant advantage is that it achieves much higher accuracy \textit{for the same number of iterations ``q''}. Therefore, to achieve a high target accuracy, the Block Krylov method often has a \textit{lower total computational cost} as it requires far fewer iterations (``q''). This trade-off is favorable for applications requiring high-precision approximations.
\end{itemize}

\RestyleAlgo{ruled}
\begin{algorithm}
{\smaller
\LinesNumbered
\SetKwInOut{Input}{Input}
\SetKwInOut{Output}{Output}  \Input{A data tensor ${\mathcal X}\in\mathbb{R}^{n_1\times n_2\times n_3}$; a tubal tensor rank $R$; Oversampling $P$ and the power iteration $q$.}  \Output{Truncated SVD: ${\mathcal X}\cong {\mathcal U}*{\mathcal S}*{\mathcal V}^\top$}
\caption{Classical randomized subspace method for computation of the Truncated SVD \cite{zhang2018randomized}}\label{ALg_1}
${\mathcal B}={\rm randn}(n_2,P+R,n_3)$;\\
${\mathcal Y}_{0}={\mathcal X}*{\mathcal B},$;\\
\For{$i=1,2,\ldots,q$}{
$\widetilde{{\mathcal Y}}_{i}=\mathcal{X}^T*\mathcal{Y}_{i-1}$;\\
${{\mathcal Y}}_{i}=\mathcal{X}*\widetilde{\mathcal{Y}}_{i}$
}
$[{\mathcal Q}_q,\sim] = {{\rm T-QR}}({\mathcal K})$;\\
$\mathcal{B} = \mathcal{Q}_q^T*\mathcal{X}$;\\
$[\widehat{\mathcal U},{\mathcal S},{\mathcal V}] = {\rm Truncated }$ $\,{\rm T-SVD}({\mathcal B},R)$;\\
${\mathcal U}={\mathcal Q}*\widehat{\mathcal U}$;
}
\end{algorithm}

\RestyleAlgo{ruled}
\begin{algorithm}
{\smaller
\LinesNumbered
\SetKwInOut{Input}{Input}
\SetKwInOut{Output}{Output}  \Input{A data tensor ${\mathcal X}\in\mathbb{R}^{n_1\times n_2\times n_3}$; a tubal tensor rank $R$; Oversampling $P$ and the power iteration $q$.}  \Output{Truncated SVD: ${\mathcal X}\cong \widehat{\mathcal U}_R*\widehat{\mathcal S}_R*\widehat{\mathcal V}^T_R$}
\caption{Proposed randomized block Krylov method for the Truncated SVD}\label{ALg_2}
${\mathcal B}={\rm randn}(n_1,P+R,n_3)$;\\
$\mathcal{K}_0 = {\mathcal X}*{\mathcal B}$;\\
\For{$i=1,2,\ldots,q$}{
${\mathcal K}_i=\mathcal{X}*\mathcal{X}^{\top}*\mathcal{K}_{i-1}$;\\
}
Generate $\mathcal{K}=[\mathcal{K}_0,\mathcal{K}_1,\ldots,\mathcal{K}_q]$;\\
$[{\mathcal Q},\sim] = {{\rm T-QR}}({\mathcal K})$;\\
$\mathcal{C}=\mathcal{Q}^T*\mathcal{X}$;\\
$[\mathcal{U}_c,\mathcal{S}_c, \mathcal{V}_c]=T-SVD(\mathcal{C})$;\\
$\widehat{\mathcal U}=\mathcal{Q}*\mathcal{U}_c$;\\
$\widehat{\mathcal U}_R=\widehat{\mathcal U}(:,1:R,:),\, \widehat{\mathcal S}_R=\mathcal{S}_c(1:R,1:R,:) ,\,\widehat{\mathcal V}_R=\mathcal{V}_c(:,1:R,:).$
}
\end{algorithm} 

\section{Theoretical results}\label{Sec:theo}
In this section, we present a detailed theoretical analysis of the proposed randomized subspace block Krylov method. A key point that needs to be highlighted is that the proposed algorithm indeed applies the randomized block Krylov method on each frontal slice of the tensor in the Fourier domain. This facilitates the derivation of the upper bound approximation obtained by the proposed algorithm.  

In the following, we present two definitions T-projection and T-orthogonal projection tensors.
\begin{defn}
A tensor $\mathcal{P} \in \mathbb{R}^{n \times n \times n_3}$ is called T-projection if $\mathcal{P}*\mathcal{P}=\mathcal{P}$.
\end{defn}
\begin{defn}
A tensor $\mathcal{P} \in \mathbb{R}^{n \times n \times n_3}$ is called T-orthogonal projection if it is T-symmetric and T-projection.
\end{defn}
From Lemma \ref{mn2} it can be observed that if $ \mathcal{A} \in \mathbb{R}^{n \times n \times n_3} $ is a T-orthogonal projection tensor, then $\mathtt{bcirc}(\mathcal{A})$ is an orthogonal projection matrix.

The following lemma shows that the spectral norm of a tensor under the T-product is orthogonally invariant.
\begin{lem}\label{mn15}
Let $ \mathcal{A} \in \Bbb R^{n_{1} \times n_{2} \times n_{3}}$, $ \mathcal{U} $ and $\mathcal{ V }$ be orthogonal tensors of dimensions $n_1 \times n_1 \times n_3$ and $n_2 \times n_2 \times n_3$ respectively. Then
\[
\parallel \mathcal{U} * \mathcal{A} *\mathcal{ V } \parallel_{2}= \parallel \mathcal{A} \parallel_{2}.
\]
\end{lem}
\begin{proof}
Since the spectral norm on matrices is unitarily invariant and $ \mathtt{bcirc} $ of an orthogonal tensor is an orthogonal matrix, we have
\[
\parallel \mathcal{U} * \mathcal{A} *\mathcal{ V }\parallel_{2} = \parallel \mathtt{bcirc}(\mathcal{U}) \mathtt{bcirc}(\mathcal{A})\mathtt{bcirc}(\mathcal{V}) \parallel_{2}=\parallel \mathtt{bcirc}(\mathcal{A}) \parallel_{2}= \parallel \mathcal{A}\parallel_{2}.
\]
\end{proof}
\begin{lem}\label{mn6}
Let  $\mathcal{A} \in \Bbb R^{n_{1} \times n_{2} \times n_{3}} $. Then, 
\[
\parallel \mathcal{A} \parallel_{2}^{2}=\parallel \mathcal{A}^\top * \mathcal{A} \parallel_{2}=\parallel \mathcal{A}* \mathcal{A}^\top \parallel_{2}.
\]
\end{lem}
\begin{proof}
From $\mathcal{A}$ is a real tensor, we deduce that the matrix $\mathtt{bcirc}(\mathcal{A})$ is a real matrix. Now, from Definition \ref{mn3} and Lemma \ref{mn2}, it follows that
\begin{eqnarray*}
\parallel \mathcal{A} \parallel_{2}^2 &=& \parallel \mathtt{bcirc}(\mathcal{A}) \parallel_{2}^2\\
&=& \parallel (\mathtt{bcirc}(\mathcal{A}))^\top  \mathtt{bcirc}(\mathcal{A}) \parallel_{2}\\
&=& \parallel \mathtt{bcirc}(\mathcal{A}^\top)  \mathtt{bcirc}(\mathcal{A}) \parallel_{2}\\
&=& \parallel \mathcal{A}^{\top}* \mathcal{A} \parallel_{2}.
\end{eqnarray*}
\end{proof}
It is not difficult to show that the summation of T-positive semidefinite tensors is a T-positive semidefinite tensor due to the linearity of T-product operator and inner product.
From [\cite{horn2012matrix}, Observation 7.1.8.], it is known that if a Hermitian matrix $A \in \mathbb{R}^{n\times n}$ is positive semideﬁnite and $C \in \mathbb{R}^{n\times n}$, then  $C^{T}AC$ is also a positive semideﬁnite matrix. The following lemma proves that the same is true for tensors under the $ T $-product.
\begin{lem}\label{mn166}
Let $ \mathcal{A} \in \Bbb R^{n\times n \times n_{3}} $ be a T-symmetric tensor and $ \mathcal{C}$ be a real-valued tensor with adequate dimension. If $\mathcal{A}$ is a T-positive semidefinite tensor, then $\mathcal{C}^{\top} * \mathcal{A} *\mathcal{C}$ is a T-positive semidefinite tensor.
\end{lem}
\begin{proof}
Since $ \mathcal{A}$ is a T-positive semidefinite tensor, from Lemma \ref{mn8}, $ \mathtt{bcirc}(\mathcal{A})$ is a symmetric positive semidefinite matrix. Now, we define $\mathcal{D}=\mathcal{C}^{\top}* \mathcal{A} *\mathcal{C} $, thus by Lemma \ref{mn2},
\[
\mathtt{bcirc}(\mathcal{D})=(\mathtt{bcirc}(\mathcal{C}))^{\top} \mathtt{bcirc}(\mathcal{A}) \mathtt{bcirc}(\mathcal{C}).
\]
Since $\mathtt{bcirc}(\mathcal{A})$ is a symmetric positive semidefinite matrix and $\mathtt{bcirc}( \mathcal{C})$ is a real  matrix, the symmetric matrix $\mathtt{bcirc}(\mathcal{D})$ is positive semidefinite. 
Hence, the T-symmetric tensor $\mathcal{D}$ is T-positive semidefinite (Lemma \ref{mn8}).
\end{proof}

\begin{lem}\label{jens:lem}
Let $\mathcal{P}$ be a real T-orthogonal projection tensor and $ \mathcal{M}$ be a real tensor with consistent dimension. For each $ s \in \Bbb N$, 
\[
\parallel \mathcal{P}*\mathcal{M}\parallel_{2}\leq \parallel \mathcal{P}*(\mathcal{M}* \mathcal{M}^{\top})^{s} *\mathcal{M}\parallel ^{\frac{1}{2s+1}}_{2}.
\]
\end{lem}
\begin{proof}
Since $ \mathcal{P}$ is an orthogonal projection tensor, the matrix $ \mathtt{bcirc}(\mathcal{P})$ is an orthogonal projection matrix. Lemmas \ref{mn2}, \ref{mn6} and Proposition 8.6 of \cite{halko2011finding} imply that
\begin{eqnarray*}
\parallel \mathcal{P}*\mathcal{M}\parallel_{2}&=&\parallel \mathtt{bcirc}(\mathcal{P}*\mathcal{M})\parallel_{2}\\
&=& \parallel \mathtt{bcirc}(\mathcal{P})\mathtt{bcirc}(\mathcal{M})\parallel_{2}\\
&\leq & \parallel \mathtt{bcirc}(\mathcal{P})\big(\mathtt{bcirc}(\mathcal{M}) (\mathtt{bcirc}(\mathcal{M}))^{\top}\big)^{s} \mathtt{bcirc}(\mathcal{M})\parallel ^{\frac{1}{2s+1}}_{2}\\
&=&\parallel \mathcal{P}*(\mathcal{M}* \mathcal{M}^{\top})^{s} *\mathcal{M}\parallel ^{\frac{1}{2s+1}}_{2}.
\end{eqnarray*}
\end{proof}
\begin{lem}\label{mn16}
Let T-symmetric $ \mathcal{A} \in \Bbb R^{n \times n  \times n} $ be a T-positive semidefinite tensor. Then $\mathcal{A}^{k}$, for $ k=0,1,2,\dots $ are T-positive semidefinite tensors.
\end{lem}
\begin{proof}
Using Lemma \ref{mn77}, we conclude that since $\mathcal{A}$ is T-positive semidefinite, its T-eigenvalues are non-negative.
By Lemma \ref{mn2}, the T-eigenvalues of $ \mathcal{A}^{k} $ are eigenvalues of $ ( \mathtt{bcirc}(\mathcal{A}))^{k}$. Hence, T-eigenvalues of $ \mathcal{A}^{k} $ are non-negative. This completes the proof.
\end{proof}
From \cite{horn2012matrix}, we know that if symmetric matrices $ A_{1},A_2,\dots,A_{m} \in \mathbb{R}^{n\times n} $ are positive semidefinite then 
\begin{eqnarray}\label{mn21}
\| \sum\limits_{i=1}^{m} A_{i}\|_{2}^{2}\geq \sum\limits_{i=1}^{m} \| A_{i} \|_{2}^{2}.
\end{eqnarray}
The following lemma shows that this also holds for T-positive semidefinite tensors.
\begin{lem}\label{mn17}
Let T-symmetric tensors $\mathcal{A}_{1}, \mathcal{A}_{2},\cdots, \mathcal{A}_{m}$ be T-positive semidefinite. Then,
\[
\| \sum\limits_{i=1}^{m} \mathcal{A}_{i}\|_{2}^{2}\geq \sum\limits_{i=1}^{m} \| \mathcal{A}_{i} \|_{2}^{2}.
\]
\end{lem}
\begin{proof}
Since tensors $\mathcal{A}_{1},\mathcal{A}_{2}, \cdots, \mathcal{A}_{m}$ are T-positive semidefinite, matrices 
$\mathtt{bcirc}(\mathcal{A}_{i})$, for $ i=1,\dots,m $, are positive semidefinite. Hence, from \eqref{mn21} we have
\begin{eqnarray*}
\| \sum\limits_{i=1}^{m} \mathtt{bcirc}(\mathcal{A}_{i})\|_{2}^{2}\geq \sum\limits_{i=1}^{m} \| \mathtt{bcirc}(\mathcal{A}_{i}) \|_{2}^{2}.
\end{eqnarray*}
This completes the proof.
\end{proof}
From Lemma 11, one can easily see that $\| \sum\limits_{i=1}^{m} \mathcal{A}_{i}\|^2_{2}\geq \| \mathcal{A}_{i}\|^2_{2}$, which we use later in our analysis.
The following theorem presents an upper bound for the error norm of the approximation obtained by the proposed algorithm for a data tensor that admits a low-rank structure, that is, of low tubal rank structure.
\begin{thm}\label{mn44}
Let $ \mathcal{X} \in \Bbb R^{n_{1} \times n_{2} \times n_{3}}$ and $ \mathcal{Q} \in \Bbb R^{n_{1} \times n_{1} \times n_{3}} $ be an orthogonal tensor that obtained from Algorithm \ref{ALg_2}. Then, 
\[
\parallel (\mathcal{I}- \mathcal{Q}* \mathcal{Q}^{\top})* \mathcal{X}\parallel_{2}\leq \parallel (\mathcal{I}- \mathcal{P}_{k})*\mathcal{Z} \parallel _{2}^{\frac{1}{2q+1}},
\]
where $\mathcal{P}_{k}= \mathcal{Q}*\mathcal{Q}^{\top} $, the power iteration parameter $q\in \Bbb N$ and $ \mathcal{Z}= [\mathcal{X}, (\mathcal{X}*\mathcal{X}^{\top})*\mathcal{X}, \ldots, (\mathcal{X}*\mathcal{X}^{\top})^{q}*\mathcal{X}] $.
\end{thm}
\begin{proof}
Let $\mathcal{X} = \mathcal{U}*\mathcal{S} * \mathcal{V}^\top$ be the T-SVD of
$\mathcal{X}$. Since $  \mathcal{P}_{k}= \mathcal{Q}*\mathcal{Q}^{\top} $, it is obvious that 
\begin{eqnarray}\label{mn23}
 \parallel (\mathcal{I}- \mathcal{Q}* \mathcal{Q}^{\top})* \mathcal{X} \parallel_{2}= \parallel (\mathcal{I}- \mathcal{P}_{k})*\mathcal{X}\parallel_{2}. 
\end{eqnarray}
Since $ \mathcal{P}_{k} $ is a T-orthogonal projection, we can see that $ \mathcal{P}= \mathcal{I}- \mathcal{P}_{k}$ is also a T-orthogonal projection tensor. Now, using Lemmas \ref{mn15} and \ref{mn6}, we have 
\begin{eqnarray}\label{mn20}
\nonumber \parallel \mathcal{P}* \mathcal{X}\parallel_{2}^{2(2q+1)}&= & \parallel \mathcal{P}*\mathcal{X}*\mathcal{X}^{\top}*\mathcal{P} \parallel_{2}^{2q+1}\\
\nonumber &=& \parallel \mathcal{U}^{T}*\mathcal{P}*\mathcal{U}*(\mathcal{S}*\mathcal{S}^{\top})*\mathcal{U}^{T}*\mathcal{P}*\mathcal{U} \parallel_{2}^{2q+1}\\
&\leq & \parallel (\mathcal{U}^{T}*\mathcal{P}*\mathcal{U})*(\mathcal{S}*\mathcal{S}^{\top})^{2q+1} *(\mathcal{U}^{T}*\mathcal{P}*\mathcal{U})\parallel_{2}.
\end{eqnarray}
in which the first inequality is obtained from Lemma \ref{jens:lem} (since $ \mathcal{U}^{T}*\mathcal{P}*\mathcal{U} $ is T-orthogonal projection). The tensor 
$ (\mathcal{U}^{T}*\mathcal{P}*\mathcal{U})*(\mathcal{S}*\mathcal{S}^{\top})^{r} *(\mathcal{U}^{T}*\mathcal{P}*\mathcal{U} )$, for all $ r \in \Bbb N $ is T-positive semidefinite because of Lemmas  \ref{mn166}. and  \ref{mn16}.
Now, let's define tensors  $\mathcal{C}$ and $ \mathcal{D} $ as follows
 \[
\mathcal{C}=({\mathcal U}^{T}*\mathcal{P}*\mathcal{U})*(\mathcal{S}*\mathcal{S}^{\top})^{(2q+1)} *(\mathcal{U}^{T}*\mathcal{P}*\mathcal{U}),
\]
 and 
\[
\mathcal{D}=(\mathcal{U}^{T}*\mathcal{P}*\mathcal{U})*(\mathcal{S}*\mathcal{S}^{\top})^{3} *(\mathcal{U}^{T}*\mathcal{P}*\mathcal{U})
+ \cdots + (\mathcal{U}^{T}*\mathcal{P}*\mathcal{U})*(\mathcal{S}*\mathcal{S}^{\top})^{2q} *(\mathcal{U}^{T}*\mathcal{P}*\mathcal{U}),
\]
Since each term of $\mathcal{D}$ is T-positive semidefinite, it can be easily shown that $\mathcal{D}$ is also semidefinite. So, from Lemma \ref{mn17}, we have 
\begin{eqnarray}\label{mn22}
\parallel \mathcal{C} \parallel_{2}\leq \parallel \mathcal{C} + \mathcal{D}\parallel_{2}. 
\end{eqnarray}
From  \eqref{mn20} -\eqref{mn22} and Lemma \ref{mn6}, we can write
\begin{eqnarray*}
\parallel \mathcal{P}* \mathcal{X}\parallel_{2}^{2(2q+1)} &\leq & \parallel ({\mathcal U}^{T}*\mathcal{P}*\mathcal{U})*(\mathcal{S}*\mathcal{S}^{\top}) *(\mathcal{U}^{T}*\mathcal{P}*\mathcal{U})\\
 & + &(\mathcal{U}^{T}*\mathcal{P}*\mathcal{U})(\mathcal{S}*\mathcal{S}^{\top})^{3} *(\mathcal{U}^{T}*\mathcal{P}*\mathcal{U})\\
&+& \cdots + (\mathcal{U}^{T}*\mathcal{P}*\mathcal{U})*(\mathcal{S}*\mathcal{S}^{\top})^{2q+1} *(\mathcal{U}^{T}*\mathcal{P}*\mathcal{U})\parallel_{2}\\
&=&\parallel \mathcal{P}*\mathcal{U}*(\mathcal{S}*\mathcal{S}^{\top}) *\mathcal{U}^{T}*\mathcal{P}+ \mathcal{P}*\mathcal{U}(\mathcal{S}*\mathcal{S}^{\top})^{3} *\mathcal{U}^{T}*\mathcal{P}\\
&+& \cdots + \mathcal{P}*\mathcal{U}*(\mathcal{S}*\mathcal{S}^{\top})^{2q+1} *\mathcal{U}^{T}*\mathcal{P}\parallel_{2}\\
&=& \parallel \mathcal{P}* [ (\mathcal{X} * \mathcal{X}^{\top})+ ( \mathcal{X} * \mathcal{X}^{\top})^{3} + \cdots + ( \mathcal{X} * \mathcal{X}^{\top})^{2q+1}]*\mathcal{P} \parallel_{2}\\
&=& \parallel \mathcal{P} *[\mathcal{X}, (\mathcal{X}*\mathcal{X}^{\top})*\mathcal{X},\dots ,(\mathcal{X}*\mathcal{X}^{\top})^{q}*\mathcal{X} ]\\
&*&
 \left[
  \begin{array}{ccc} 
 \mathcal{X}^{\top}  \\
((\mathcal{X}*\mathcal{X}^{\top})*\mathcal{X})^{\top} \\ 
\vdots \\
   ((\mathcal{X}*\mathcal{X}^{\top})^{q}*\mathcal{X})^{\top} 
   \end{array} \right] * \mathcal{P}  \parallel_{2}\\
   &=& \parallel \mathcal{P} *[\mathcal{X}, (\mathcal{X}*\mathcal{X}^{\top})*\mathcal{X},\dots ,(\mathcal{X}*\mathcal{X}^{\top})^{q}*\mathcal{X} ] \parallel_{2}^{2}\\
&=& \parallel \mathcal{P} * \mathcal{Z} \parallel_{2}^{2}.
\end{eqnarray*}
From \eqref{mn23}, we have
\[
\parallel (\mathcal{I}- \mathcal{Q}* \mathcal{Q}^{\top})* \mathcal{X} \parallel_{2}=\parallel (\mathcal{I}-\mathcal{P}_{k})* \mathcal{X}\parallel_{2}\leq
\parallel (\mathcal{I}-\mathcal{P}_{k})* \mathcal{Z}\parallel_{2}^{\frac{1}{2q+1}}. 
\]
\end{proof}
Let $ \mathcal{X} = \mathcal{U} * \mathcal{S} * \mathcal{V}^\top$ be the T-SVD of
$\mathcal{X} \in \mathbb{R}^{I_1 \times I_2 \times I_3}$. Now, we have
\begin{eqnarray}\label{mn1}
(\mathcal{X}* \mathcal{X}^\top)^{q}*\mathcal{X}  &=& \mathcal{U}* (\mathcal{S}*\mathcal{S}^{\top})^{q} *\mathcal{S} * \mathcal{V}^{\top},
\end{eqnarray}
where $ \mathcal{X}^{q} $ denotes $ \underbrace{\mathcal{X} *\mathcal{X}* \dots * \mathcal{X}}_\text{q times} $. By Algorithm \ref{ALg_2}, we can write
\begin{eqnarray*}\label{mn27}
\nonumber\mathcal{K}=[\mathcal{K}_0,\mathcal{K}_1,\ldots,\mathcal{K}_q]&=&[\mathcal{X}*\mathcal{B}, (\mathcal{X}*\mathcal{X}^{\top})*\mathcal{X}*\mathcal{B},\ldots , (\mathcal{X}*\mathcal{X}^{\top})^{q}*\mathcal{X}*\mathcal{B}]\\
\nonumber &=& [\mathcal{X}, (\mathcal{X}*\mathcal{X}^{\top})*\mathcal{X}, \ldots, (\mathcal{X}*\mathcal{X}^{\top})^{q}*\mathcal{X}]\\&&*
\left[
 \begin{array}{ccccc} 
\mathcal{B} & && &  \\ 
 & \mathcal{B} & & & \\
   & &\ddots  & &\\
  && & & \mathcal{B}
   \end{array} \right],  
\end{eqnarray*}
in which $ \mathcal{Z}= [\mathcal{X}, (\mathcal{X}*\mathcal{X}^{\top})*\mathcal{X}, \ldots, (\mathcal{X}*\mathcal{X}^{\top})^{q}*\mathcal{X}] \in \Bbb R^{I_{1}\times I_{2}(q+1)  \times I_{3}}$. Now, as a result of relation \eqref{mn1}, we have
\begin{eqnarray}\label{mn24}
\mathcal{Z}&=&[\mathcal{X}, (\mathcal{X}*\mathcal{X}^{\top})*\mathcal{X}, \ldots, (\mathcal{X}*\mathcal{X}^{\top})^{q}*\mathcal{X}]\\
\nonumber &=& [\mathcal{U} * \mathcal{S} * \mathcal{V}^\top , \mathcal{U}* (\mathcal{S}*\mathcal{S}^{\top}) *\mathcal{S} * \mathcal{V}^{\top},\ldots , \mathcal{U}* (\mathcal{S}*\mathcal{S}^{\top})^{q} *\mathcal{S} * \mathcal{V}^{\top}\\
\nonumber &=&\mathcal{U}* [\mathcal{S},(\mathcal{S}*\mathcal{S}^{\top}) *\mathcal{S}, \ldots , (\mathcal{S}*\mathcal{S}^{\top})^{q}*\mathcal{S}] *
\left[
 \begin{array}{ccccc} 
\mathcal{V}^{\top} & && &  \\ 
 &\mathcal{V}^{\top} & & & \\
   & &\ddots  & &\\
  && & & \mathcal{V}^{\top}
   \end{array} \right]. 
\end{eqnarray}
Consider the T-SVD for tensor the $ [\mathcal{S},(\mathcal{S}*\mathcal{S}^{\top}) *\mathcal{S}, \ldots, (\mathcal{S}*\mathcal{S}^{\top})^{q}*\mathcal{S}]$, thus
\begin{eqnarray}\label{mn25}
[\mathcal{S},(\mathcal{S}*\mathcal{S}^{\top}) *\mathcal{S}, \ldots, (\mathcal{S}*\mathcal{S}^{\top})^{q}*\mathcal{S}] = \tilde{\mathcal{U}}*\tilde{\mathcal{S}}*\tilde{\mathcal{V}}^{\top},
\end{eqnarray}
where $ \tilde{\mathcal{U}} $ and $\tilde{\mathcal{V}}$ are orthogonal tensors and $ \tilde{\mathcal{S}} $ is an f-diagonal tensor, in which the entries on the first frontal slice are the singular values of the tensor $ [\mathcal{S},(\mathcal{S}*\mathcal{S}^{\top}) *\mathcal{S}, \ldots, (\mathcal{S}*\mathcal{S}^{\top})^{q}*\mathcal{S}]$. In other words, they are the square root of the eigenvalues of the following tensor
\[
\mathcal{S}*\mathcal{S}^{\top} + (\mathcal{S}*\mathcal{S}^{\top})^{3}+ \cdots + (\mathcal{S}*\mathcal{S}^{\top})^{2q+1}.
\] 
Now, relations \eqref{mn24} and \eqref{mn25} conclude that 
\begin{eqnarray}\label{mn26}
\nonumber \mathcal{Z}&=&\mathcal{U}*\tilde{\mathcal{U}}* \tilde{\mathcal{S}}* \tilde{\mathcal{V}}^{\top} *
\left[
 \begin{array}{ccccc} 
\mathcal{V} & && &  \\
 &\mathcal{V} & & & \\
& &\ddots  & &\\
&& & & \mathcal{V}
\end{array} \right]^{\top}, \\
&=& \widehat{\mathcal{U}}* \tilde{\mathcal{S}}*\widehat{\mathcal{V}}^{\top},
\end{eqnarray}
where $ \widehat{\mathcal{U}}= \mathcal{U}*\tilde{\mathcal{U}} \in \Bbb R^{I_{1}\times I_{1} \times I_{3}}$ and 
\[
\widehat{\mathcal{V}}^{\top}=\tilde{\mathcal{V}}^{\top} *
\left[
 \begin{array}{ccccc} 
\mathcal{V} & && &  \\
 &\mathcal{V} & & & \\
   & &\ddots  & &\\
  && & & \mathcal{V}
   \end{array} \right]^{\top} \in \Bbb R^{I_{2}\times I_{2}(q+1)\times I_{3}}.
\] 
\begin{thm}
Let $ \mathcal{X} \in \Bbb R^{n_{1} \times n_{2} \times n_{3}}$ and $ \mathcal{Q} \in \Bbb R^{n_{1} \times n_{1} \times n_{3}} $ be an orthogonal tensor that obtained from Algorithm \ref{ALg_2}. Then, by notations in \eqref{mn24} and \eqref{mn26}
\[
\parallel \mathcal{X}- \mathcal{Q}* \mathcal{Q}^{\top}* \mathcal{X}\parallel_{2}\leq \sigma_{\max}^{\frac{1}{2q+1}}(\tilde{\mathcal{S}}),
\]
in which power parameter $q \in \Bbb N$ and tensor $\tilde{\mathcal{S}}$ is mentioned in \eqref{mn26}.
\end{thm}
\begin{proof}
Since $\mathcal{P}=\mathcal{I}- \mathcal{P}_{k} $ is a T-orthogonal projection tensor, $ \mathtt{bcirc}(\mathcal{P})$ is an orthogonal projection matrix. Hence, the eigenvalues of idempotent matrix $ \mathtt{bcirc}(\mathcal{P}) $ are $ 0 $ or $ 1 $. From Definition \eqref{mn45}, it follows that the T-eigenvalues of $\mathcal{P}$ are $0$ or $1$.
Now, Lemma \ref{mn6} proves that $\parallel \mathcal{I}- \mathcal{P}_{k}\parallel_{2}=1$, where $  \mathcal{P}_{k}= \mathcal{Q}*\mathcal{Q}^{\top} $. Hence, from Lemma \ref{mn15}, Theorem \ref{mn44} and Lemma \ref{mn43}, we have
\begin{eqnarray*}
\parallel \mathcal{X}- \mathcal{Q}* \mathcal{Q}^{\top}* \mathcal{X}\parallel_{2}&\leq &\parallel (\mathcal{I}- \mathcal{P}_{k})*\mathcal{Z} \parallel _{2}^{\frac{1}{2q+1}}\\
&\leq & (\parallel \mathcal{I}- \mathcal{P}_{k}\parallel_{2}\parallel \mathcal{Z}\parallel_{2})^{\frac{1}{2q+1}}\\
&= & \parallel \mathcal{Z}\parallel_{2}^{\frac{1}{2q+1}}\\
&=&\parallel \widehat{\mathcal{U}}* \tilde{\mathcal{S}}*\widehat{\mathcal{V}}^{\top} \parallel_{2}^{\frac{1}{2q+1}}\\
&=& \parallel \tilde{\mathcal{S}} \parallel_{2}^{\frac{1}{2q+1}},
\end{eqnarray*}
and this completes the proof.
\end{proof}

It is not difficult to show that if $A^{T}B =O$, in which $O$ is zero matrix, then $||A+B||_{F}^{2}=||A||_{F}^{2}+ ||B||_{F}^{2}$, where $ A $ and $ B $ are matrices of size $ m\times n $. The following lemma extends this result to tensors under the T-product.
\begin{lem}\label{mn440}
Let $ \mathcal{A} $ and $ \mathcal{B} $ be two real tensors of dimension $ n_{1}\times n_{2} \times n_{3}$. If $\mathcal{A}^{\top}*\mathcal{B} =\mathcal{O} $, where $\mathcal{O}$ is a zero tensor of dimension $n_2 \times n_2 \times n_3$, then
\[
\parallel \mathcal{A} + \mathcal{B}\parallel_{F}^{2}=\parallel\mathcal{A}\parallel_{F}^{2} +\parallel\mathcal{B}\parallel_{F}^{2}.
\]
\end{lem}
\begin{proof}
From \cite{cao2023some}, we know that if  $\mathcal{A}^{\top}*\mathcal{B} =\mathcal{O} $ then $\overline{A}^{\top}\overline{B} =O $. 
Now using relation \eqref{eq_fou} and above discussion, we have
\begin{eqnarray*}
\parallel\mathcal{A} + \mathcal{B} \parallel_{F}^{2}&=& \frac{1}{n_3}\parallel \overline{A+B}\parallel_{F}^{2}\\
&=& \frac{1}{n_3}\parallel\overline{A}+\overline{B}\parallel_{F}^{2}\\
&=& \frac{1}{n_3}\parallel\overline{A}\parallel_{F}^{2} +\frac{1}{n_3}\parallel \overline{B}\parallel_{F}^{2}  \\
&=& \parallel\mathcal{A}\parallel_{F}^{2} +\parallel\mathcal{B}\parallel_{F}^{2}.
\end{eqnarray*}
\end{proof}
Suppose that $ \mathcal{A} \in \Bbb R^{n_{1}\times n_{2}\times n_{3}} $ and $\mathcal{A}=\mathcal{U}* \mathcal{S}*\mathcal{V}^{\top}$ is T-SVD of $ \mathcal{A} $. Let $ \mathcal{A}_{k}$ be the best rank-k approximation for $ \mathcal{A} $, then 
\begin{eqnarray}\label{mn37}
\mathcal{A}=\mathcal{A}_{k} +\mathcal{A}_{k,\perp},
\end{eqnarray}
in which $\mathcal{A}_{k}=\mathcal{U}_{k}* \mathcal{S}_{k}*\mathcal{V}_{k}^{\top}$, $\mathcal{U}_{k}=\mathcal{U}(:,1:k,:)$, $ \mathcal{S}_{k}=\mathcal{S}(1:k,1:k,:)$, $ \mathcal{V}_{k}=\mathcal{V}(:,1:k,:)$. Also, we have $\mathcal{A}_{k,\perp}=\mathcal{U}_{k,\perp}* \mathcal{S}_{k,\perp}*\mathcal{V}_{k,\perp}^{\top}$, $\mathcal{U}_{k,\perp}=\mathcal{U}(:,k+1:n_{1},:) $, $\mathcal{S}_{k,\perp}=\mathcal{S}(k+1:n_{1},k+1:n_{2},:)$, and $ \mathcal{V}_{k,\perp}=\mathcal{V}(:,k+1:n_{2},:)$.
\begin{lem}\label{mn38}
Let $ \mathcal{Q} $ be an orthogonal basis for $ \mathcal{K} $ and $ \widehat{\mathcal{U}}_{k} $ be the top $k$ vertical laterals of $ \widehat{\mathcal{U}}$ for the Algorithm \ref{ALg_2}. Then, 
\[
\mathcal{A}-\widehat{\mathcal{U}}_{k}*\widehat{\mathcal{U}}_{k}^{\top} *\mathcal{A}=\mathcal{A}- \mathcal{Q} *(\mathcal{Q}^{\top}*\mathcal{A})_{k}.
\]
\end{lem}
\begin{proof}
By Algorithm \ref{ALg_2}, $ \widehat{\mathcal{U}}_{k} $ involves  the top $k$ vertical laterals of $\widehat{\mathcal{U}}=\mathcal{Q}*\mathcal{U}_{C}$, where $ \mathcal{U}_{C}$ is obtained from applying T-SVD on $ \mathcal{C}=\mathcal{Q}^{\top}*\mathcal{A} $. Hence, the top $k$ vertical laterals of $\mathcal{U}_{C} $, (i.e. $ \mathcal{U}_{C,k} $), span the same range as $ \mathcal{C}_{k}$. Thus, we have
\begin{eqnarray*}
\mathcal{A}-\widehat{\mathcal{U}}_{k}* \widehat{\mathcal{U}}_{k} ^{\top} * \mathcal{A}&=& \mathcal{A}- (\mathcal{Q}*\mathcal{U}_{C})_{k} * ((\mathcal{Q}*\mathcal{U}_{C})_{k})^{\top}* \mathcal{A}\\
&=& \mathcal{A}- \mathcal{Q}* \mathcal{U}_{C,k} *\mathcal{U}_{C,k}^{\top} *\mathcal{Q}^{\top}*\mathcal{A}\\
&=& \mathcal{A}-\mathcal{Q}*\mathcal{C}_{k}*\mathcal{C}_{k}^{\dagger}*\mathcal{C}\\
&=& \mathcal{A}- \mathcal{Q}*\mathcal{C}_{k}*\mathcal{C}_{k}^{\dagger}*(\mathcal{C}_{k}+\mathcal{C}_{k,\perp})\\
&=& \mathcal{A}- \mathcal{Q}*\mathcal{C}_{k}\\
&=& \mathcal{A}- \mathcal{Q}*(\mathcal{Q}^{\top}*\mathcal{A})_{k}.
\end{eqnarray*}
\end{proof}
Lemma \ref{mn38} concludes that the best rank-k approximation for $\mathcal{A}$ from ${\rm range}(\mathcal{K})$ is
$ \mathcal{Q}*(\mathcal{Q}^{\top}*\mathcal{A})_{k} $. In other words
\begin{eqnarray}\label{mn39}
\parallel \mathcal{A}- \mathcal{Q}*(\mathcal{Q}^{\top}*\mathcal{A})_{k}\parallel_{F}^{2}= \min_{T-rank(\mathcal{Y})\leq k}\parallel \mathcal{A}- \mathcal{Q}*\mathcal{Y}\parallel_{F}^{2}.
\end{eqnarray}

\begin{thm}
Let $ \mathcal{Q} $ be an orthogonal basis for $ \mathcal{K} $ and $ \widehat{\mathcal{U}}_{k} $ be the top $k$ vertical laterals of $ \widehat{\mathcal{U}}$ for the Algorithm \ref{ALg_2}. Then, 
\begin{eqnarray*}
\parallel \mathcal{A}-\widehat{\mathcal{U}}_{k}*\widehat{\mathcal{U}}_{k}^{\top} *\mathcal{A} \parallel_{F}^{2}\leq \parallel\mathcal{A}_{k}- \mathcal{Q}*\mathcal{Q}^{\top}*\mathcal{A}_{k} \parallel_{F}^{2} + \parallel\mathcal{A}_{k,\perp}\parallel_{F}^{2}.
\end{eqnarray*}
\end{thm}
\begin{proof}
Lemma \ref{mn440}, Lemma \ref{mn38}, and relation \eqref{mn41} together imply that
\begin{eqnarray*}
\parallel\mathcal{A}-\widehat{\mathcal{U}}_{k}*\widehat{\mathcal{U}}_{k}^{\top} *\mathcal{A} \parallel_{F}^{2}&=& \parallel\mathcal{A}- \mathcal{Q} *(\mathcal{Q}^{\top}*\mathcal{A})_{k} \parallel_{F}^{2}\\
&\leq &\parallel\mathcal{A}- \mathcal{Q} *\mathcal{Q}^{\top}*\mathcal{A}_{k}\parallel_{F}^{2} \\
&=&\parallel \mathcal{A}_{k}+ \mathcal{A}_{k,\perp}- \mathcal{Q} *\mathcal{Q}^{\top}*\mathcal{A}_{k}\parallel_{2}^{F}\\
&=& \parallel \mathcal{A}_{k}- \mathcal{Q} *\mathcal{Q}^{\top}*\mathcal{A}_{k}\parallel_{F}^{2}+\parallel \mathcal{A}_{k,\perp} \parallel_{F}^{2},
\end{eqnarray*}
the first inequality is obtained from \eqref{mn39} and the second equality is obtained from\eqref{mn37}.
\end{proof}

\section{Experimental results}\label{sec:experi}
This section presents the experimental results. We conducted experiments to compress images using MATLAB on a Laptop computer with Intel(R) Core(TM) i7-10510U CPU processor and 16 GB memory. 

The PSNR of two images ${\mathcal X}$ and ${\mathcal Y}$ is defined as
\[
{\rm PSNR}=10\log _{10}\left({\frac{255^2}{{\rm MSE}}}\right)\,{\rm dB},
\]
where ${\rm MSE}=\frac{||{\mathcal X}-{\mathcal Y}||_F^2}{{\rm num}({\mathcal X})}$ and ''num(${\mathcal X}$)'' stands for the number of elements of the data tensor ${\mathcal X}$.
The relative error is also defined as 
\[
e(\widetilde{\mathcal X})=\frac{\|\mathcal X-\widetilde{\mathcal X}\|_F}{\|\mathcal X\|_F},
\]
where $\mathcal X$ is the original tensor and $\widetilde{\mathcal X}$ is the approximated tensor. 

\begin{exa}\label{EX:1} (synthetic data) In this example, we consider synthetic data tensors of low tubal rank structure. Let us generate a tensor of dimension $n\times n\times n$ of low tubal rank where the singular values of the frontal slices are all the same as 
$S={\rm diag}(\sigma_1,\sigma_2,\ldots,\sigma_n)
$
where
\begin{itemize}
\item Case 1: $\sigma_m=\frac{1}{m^5}$ for $m=1,2,\ldots,n$

\item Case 2: $\sigma_m=\frac{1}{m^6}$ for $m=1,2,\ldots,n$

\item Case 3: $\sigma_m=(0.5)^m$ for $m=1,2,\ldots,n$
\end{itemize}
The proposed randomized algorithm and the baseline (Algorithm \ref{ALg_1}) are applied to the mentioned data tensor with the parameters $P=5,\,q=2$ and $R=45$. The relative errors and running times of the algorithms are reported in Figures \ref{fig1:ex1}-\ref{fig3:ex1}. We should point out that in Algorithm \ref{ALg_2}, in Line 7, the factor $\mathcal{Q}$ was truncated to $k+p$ term as our simulations showed that there is no difference between truncation and the full factor for this synthetic data, and our reported running times are for this case. If we consider the full factor $\mathcal{Q}$, the running time of Algorithm \ref{ALg_2} will be higher than Algorithm \ref{ALg_1}. What we see from our numerical experiments is that the proposed randomized algorithm can provide estimates with better accuracy compared to the baseline, with a slightly higher running time. As was discussed in Section \ref{Sec:prop}, the superior performance of Algorithm \ref{ALg_2}, is because it can capture a much richer spectral range, leading to faster convergence and higher accuracy for the same number of tensor operations.

We also made a sensitivity analysis on the oversampling and power iteration parameters. Our observation was that for tensors with frontal slices with fast singular value decay, the small oversampling parameter, e.g. \(p = 5, 10, 20\) is sufficient, often independent of the target rank \(R\). Additionally, for the power iteration parameter, small integers, e.g., \(q = 0, 1, 2, 3\), were sufficient and rarely exceeded \(q = 4 \) due to numerical stability concerns.

\begin{figure}
\begin{center}
\includegraphics[width=0.7\linewidth]{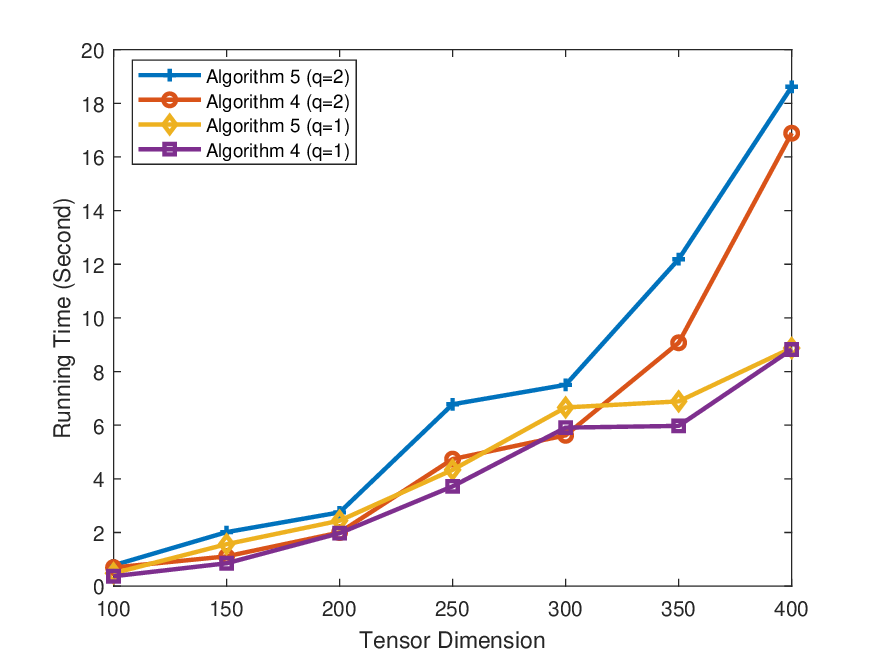}
\includegraphics[width=0.7\linewidth]{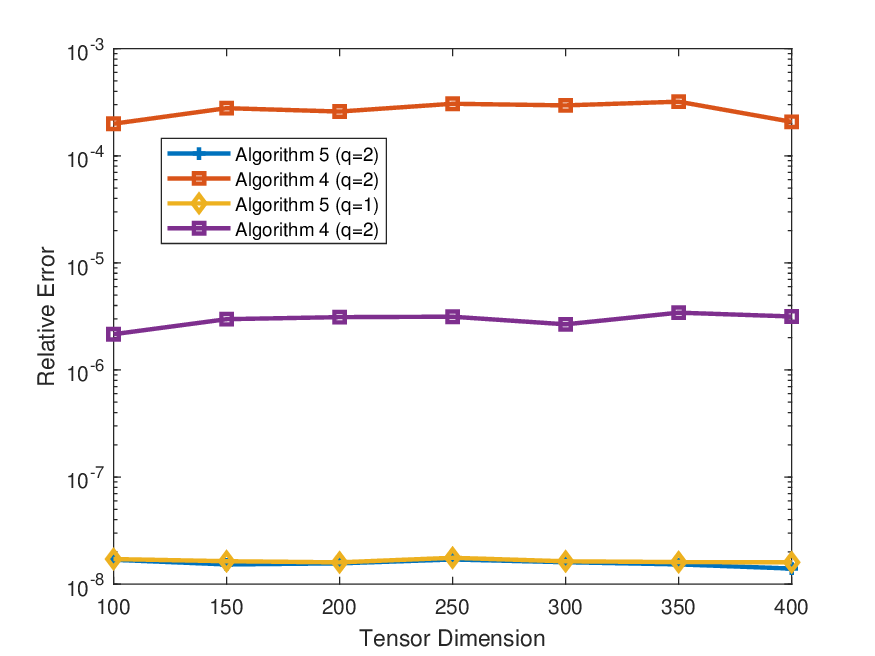}
\caption{\small{Relative error and computing time for Algorithms \ref{ALg_1} and \ref{ALg_2}} for Example \ref{EX:1}, Case 1.}\label{fig1:ex1}
\end{center}
\end{figure}

\begin{figure}
\begin{center}
\includegraphics[width=0.7\linewidth]{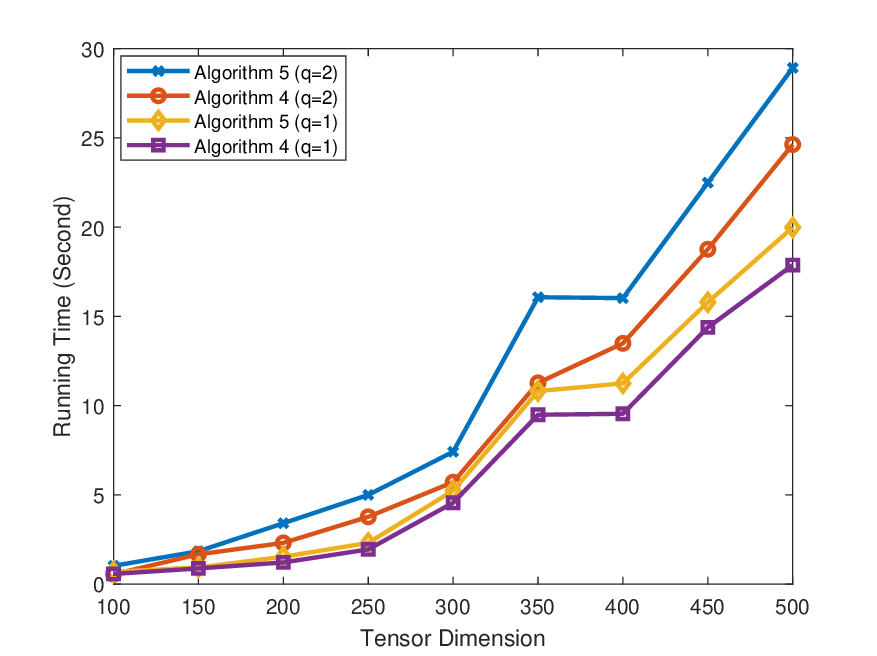}
\includegraphics[width=0.7\linewidth]{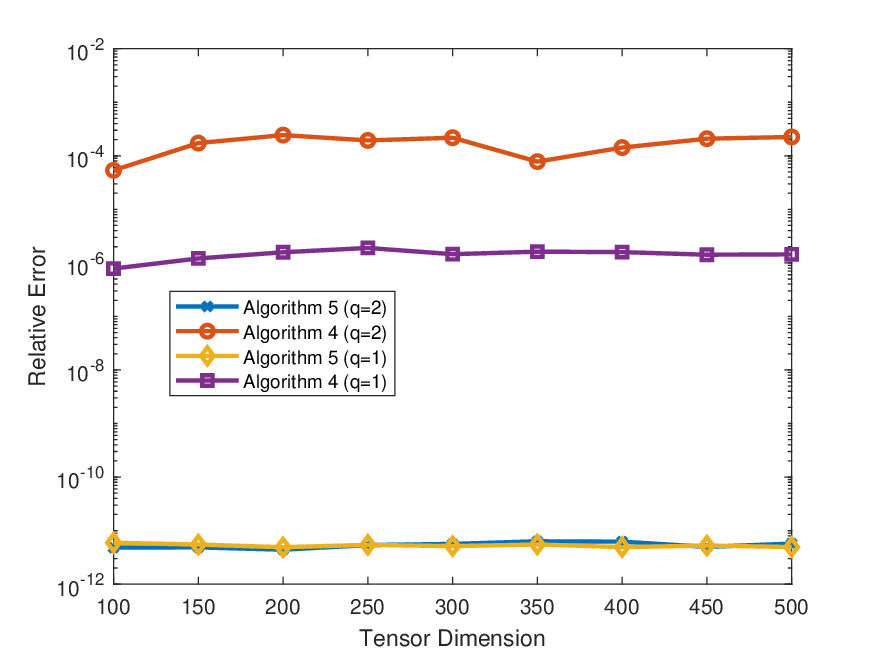}
\caption{\small{Relative error and computing time for Algorithms \ref{ALg_1} and \ref{ALg_2}} for Example \ref{EX:1}, Case 2.}\label{fig2:ex1}
\end{center}
\end{figure}

\begin{figure}
\begin{center}
\includegraphics[width=0.7\linewidth]{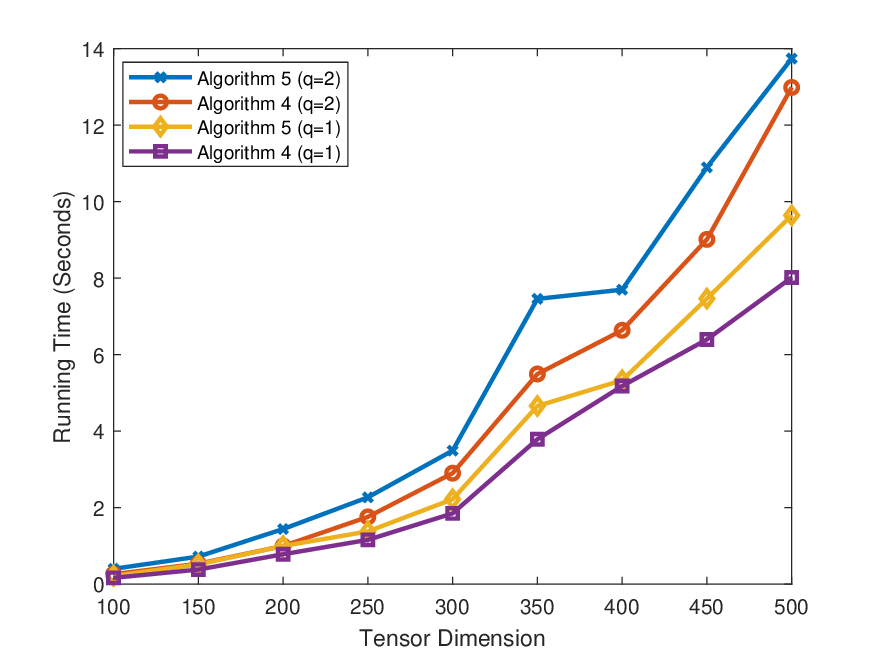}
\includegraphics[width=0.7\linewidth]{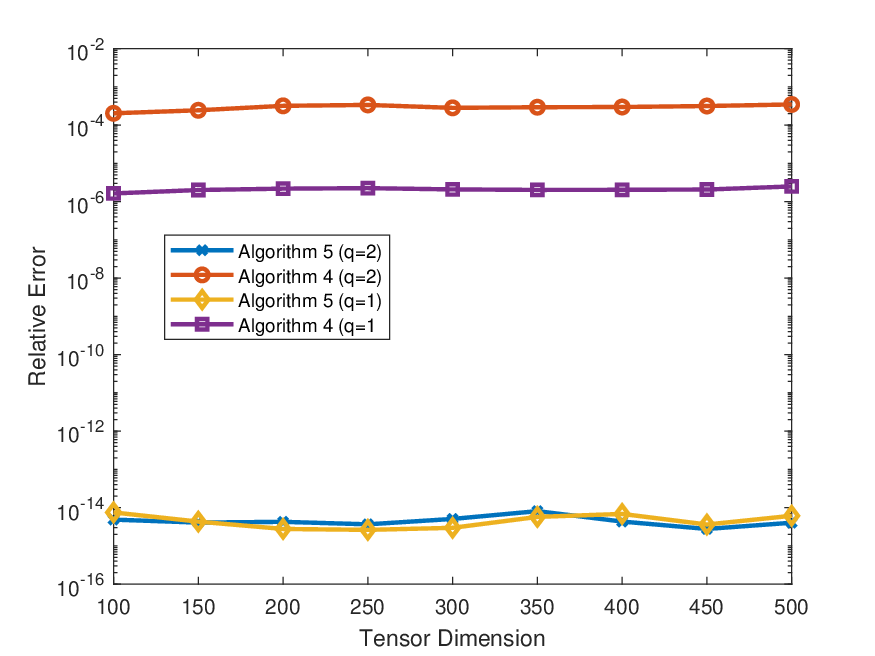}
\caption{\small{Relative error and computing time for Algorithms \ref{ALg_1} and \ref{ALg_2}} for Example \ref{EX:1}, {Case 3}.}\label{fig3:ex1}
\end{center}
\end{figure}

\end{exa}

\begin{exa}\label{EX:2}
(Image compression) In this example, we consider the application of the proposed randomized block Krylov subspace method for compressing color images. We consider three color images ``peppers'', ``baboon'' and ``house'' depicted in Figure \ref{fig:ex2}. All images have size $256\times 256\times 3$. Then, the proposed randomized approach and the benchmark Algorithm \ref{ALg_1} were applied to the three images using $P=5,\,q=2$ and the tubal rank $R=25$. The reconstructed images computed by them are shown in Figure \ref{fig:ex2}. The proposed algorithm can achieve marginally better results than the baseline. At the same time, its running time is a little higher due to dealing with a block of matrices (the tensor $\mathcal{K}$ in Algorithm \ref{ALg_2}) instead of only one single matrix (the tensor $\mathcal{B}$ in Line 7 of Algorithm \ref{ALg_1}). This requires extra computations. The simulation results indicate that the proposed randomized algorithm yields better results than the baseline. The sensitivity analysis on the oversampling and power iteration parameters was the same as the first example. More precisely, for the oversampling parameter, we achieved smooth improvement for 1 to 10. For the power iteration from $q=1$ to, $q=2,$ we saw a significant improvement, while for the $q\geq 2$ the improvement was not so much.

\begin{figure}
\begin{center}
\includegraphics[width=1\linewidth]{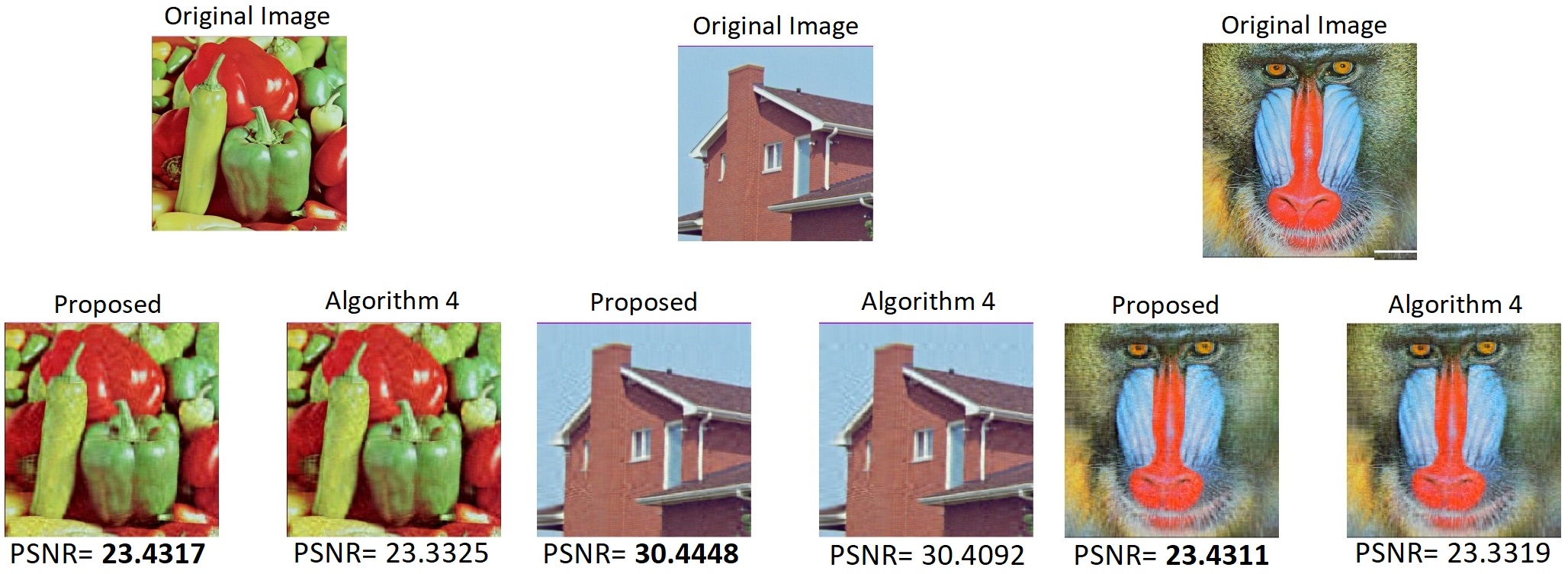}
\caption{\small{Comparing the PSNR of compressed images using Algorithms \ref{ALg_1} and \ref{ALg_2}} for Example \ref{EX:2}.}\label{fig:ex2}
\end{center}
\end{figure}

\end{exa}

\begin{exa}\label{EX:3} (Image completion) In this experiment, we apply the proposed randomized block Krylov subspace method for the data completion task. We use the methodology proposed in \cite{ahmadi2023fast,ahmadi2024randomized}, where a low-rank approximation is required in the completion process. To ensure the paper is self-contained, we will briefly describe it. The image completion can be formulated as an optimization problem:
\begin{equation}\label{MinRankCompl2}
\begin{array}{cc}
\displaystyle\min_{{\mathcal X}} & {\|{{ P}_{{\mathcal{T}}} }({{\mathcal X}})-{{ P}_{{\mathcal T}} }({{\mathcal M}})\|^2_F},\\
\textrm{s.t.} & {\rm Rank}({\mathcal X})=R,\\
\end{array}
\end{equation}
where ${\mathcal M}$ is the exact data image, the operator ${{\bf P}_{\mathcal T} }\left( {\mathcal X} \right)$ projecting the data tensor $\mathcal X$ onto the observation index tensor 
$\mathcal{T}$ is defined as
\[
{{P}_{\mathcal T} }\left( {\mathcal X} \right) = \left\{ \begin{array}{l}
{x_{{i_1},{i_2}, \ldots ,{i_N}}}\,\,\,\,\,\,\left( {{i_1},{i_2}, \ldots ,{i_N}} \right) \in \mathcal{T}. \\
0\,\,\,\,\,\,\,\,\,\,\,\,\,\,\,\,\,\,\,\,\,\,\,\,\,\,\,\,\,\,\,{\rm Otherwise}.
\end{array} \right.
\]
Using an auxiliary variable ${{\mathcal C}}$, the optimization problem \eqref{MinRankCompl2} can be solved more conveniently by the following reformulation
\begin{align}\label{MinRankCompl3}
\nonumber
\min_{{\mathcal X},\,{\mathcal C}}  {\|{{\mathcal X}}-{{\mathcal C}}\|^2_F},\\
\nonumber
\textrm{s.t.}  \,\,\,{\rm Rank}({\mathcal X})=R,\,\\
{{P}_{{\mathcal T}} }({{\mathcal C}})={{ P}_{{\mathcal T}} }({{\mathcal M}})
\end{align}
thus we can solve the minimization problem \eqref{MinRankCompl3} over variables ${\mathcal X}$ and ${\mathcal C}$. This is indeed an ALS ({Alternating Least Squares}) method that alternates between fixing one factor matrix and optimizing the other, effectively reducing the optimization problem into smaller subproblems. So, the solution to the minimization problem \eqref{MinRankCompl2} can be approximated by the following iterative procedures:
\begin{equation*}\label{Step1}
{\mathcal X}^{(n)}\leftarrow \mathcal{L}({\mathcal C}^{(n)}),
\end{equation*}
\begin{equation*}\label{Step2}
{\mathcal C}^{(n+1)}\leftarrow{\mathcal  T}\oast{\mathcal M}+({\mathbf 1}-{\mathcal T})\oast{\mathcal X}^{(n)},
\end{equation*}
where $\mathcal{L}$ is an operator to compute a low-rank tubal approximation of the data tensor ${\mathcal C}^{(n)}$, $\oast$ is a symbol to denote the element-wise multiplication of two tensors, and ${\mathbf 1}$ is a tensor whose all components are equal to one. Here, we apply the proposed Algorithm \ref{ALg_2} and the baseline \ref{ALg_1} for the operator $\mathcal{L}$. For this, we consider Kodak dataset\footnote{https://r0k.us/graphics/kodak/}, some of its samples are depicted in Figure \ref{fig1:ex31}. We use three benchmark images: ``Kodim01'', ``Kodim02'' and ``Kodim03'' from Kodak dataset. The size of the images is $512\times 768\times 3$. We randomly remove 70\% of the pixels and attempt to recover the mixing pixels with a tubal rank of 50. The observed, original, and recovered images are displayed in Figure \ref{fig:ex3}. The oversampling parameter $P=10$ and power iteration $q=2$ were used in Algorithms \ref{ALg_1} and \ref{ALg_2}. As can be seen, the proposed algorithm can successfully recover the missing pixels. The computing time of Algorithm \ref{ALg_1} was approximately 40 seconds, while it was 47 seconds for Algorithm \ref{ALg_2}, indicating that their computational complexity behaves similarly. To further evaluate the efficiency of the proposed approach for reconstructing incomplete images with structured missing patterns, we considered six distinct pattern types: Ocular lines (rows/columns), Ocular lines with a higher density of lines, Ocular circles, Ocular circles with a higher density of circles, Ocular circles with larger circles, and Text masks. These patterns, applied to the ``Kodim03'' image, are illustrated in Figure~\ref{figstru_miss:ex3}. The image completion was performed using a tubal rank of $R=30$ for 200 iterations. The resulting recovered images are also shown in Figure~\ref{figstru_miss:ex3}.

Furthermore, we examined the proposed technique under very high missing ratios of 90\%, 95\%, and 98\%. The results for these extreme cases are reported in Figure~\ref{fig:extreme_case}. While the same parameters were used for the 90\% case (tubal rank $R=15$ and 100 iterations), a lower tubal rank $R=15$ and 500 iterations were necessary to achieve satisfactory results for the 95\% and 98\% missing ratios.
\begin{figure}
\begin{center}
\includegraphics[width=0.7\linewidth]{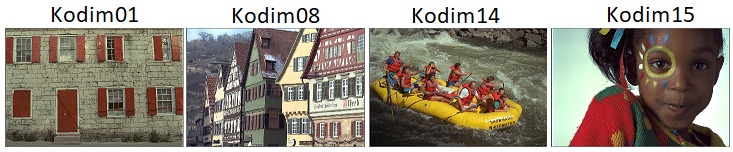}\caption{\small{Some random samples of the kodak dataset}}\label{fig1:ex31}
\end{center}
\end{figure}

Three completion methods as baselines, including: Bayesian tensor CPD (Canonical Polyadic decomposition) completion\footnote{https://github.com/qbzhao/BCPF} \cite{zhao2015bayesian}, tensor ring completion\footnote{https://github.com/HuyanHuang/Tensor-completion-via-tensor-ring-decomposition} \cite{huang2020provable} and tensor train completion\footnote{https://github.com/zhaoxile/} \cite{ding2019low} were used. The comparison was made for the ``kodim03'' image and for high missing ratios. The recovered images by these algorithms are shown in Figure \ref{fig:comp}. For all of these algorithms, the default parameters and tensor ranks are used in the demo files of their GitHub pages.
Our experiments show the method is highly effective for image completion, performing robustly under both random and structured missing patterns. Comparing our results with the baselines for 95\%, and 98\% shows significant improvement. The Bayesian tensor CPD and tensor train completion were very slow, but the tensor ring completion algorithm was faster. The running time of our algorithm was comparable to the tensor ring completion, but it was approximately 4-5 times faster. {
For a fair comparison, we also included two randomized tensor completion methods~\cite{ahmadi2020randomized,ahmadi2025efficient}. Using the same baseline images—\texttt{Kodim01}, \texttt{Kodim02}, and \texttt{Kodim03}—we reshaped them into tensors of size $8 \times 8 \times 8 \times 8 \times 8 \times 4 \times 9$ and applied a tensor ring rank of $(10,10,10,10,10,10,10)$. With the method from~\cite{ahmadi2020randomized}, we obtained PSNRs of $(27.45\,{\rm dB},\,22.65\,{\rm dB},\,28.11\,{\rm dB})$\footnote{The first, second, and third components are the results for \texttt{Kodim01}, \texttt{Kodim02}, and \texttt{Kodim03}, respectively.} and computation times of $(30.34,\,31.49,$ $\,29.98)$\footnote{The first, second, and third components are the results for \texttt{Kodim01}, \texttt{Kodim02}, and \texttt{Kodim03}, respectively.}(in seconds). The method in~\cite{ahmadi2025efficient} achieved PSNRs of $(29.94\,{\rm dB},$ 
$\,24.41\,{\rm dB},\,30.19\,{\rm dB})$ and computation times of $(12.56\,,13.88,\,13.22)$(in seconds). Our approach produced results very similar to those of~\cite{ahmadi2025efficient} while being slightly faster. Moreover, compared to the method in~\cite{ahmadi2023fast}, our algorithm was significantly faster and more efficient. These findings further confirm that the proposed randomized algorithm is highly effective for completion tasks.}

We should also highlight that we observed similar robustness in the sensitivity analysis for the oversampling and power iteration parameters. This demonstrates that the algorithm's performance is not highly sensitive to the specific choices for these parameters.

\begin{figure}
\begin{center}
\includegraphics[width=0.7\linewidth]{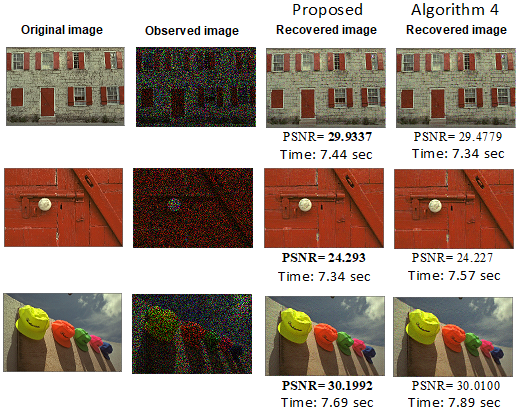}
\caption{\small{Relative error and computing time for Algorithms \ref{ALg_1} and \ref{ALg_2}} for example \ref{EX:3}.}\label{fig:ex3}
\end{center}
\end{figure}

\begin{figure}
    \centering
    \includegraphics[width=1.5\linewidth, angle=90]{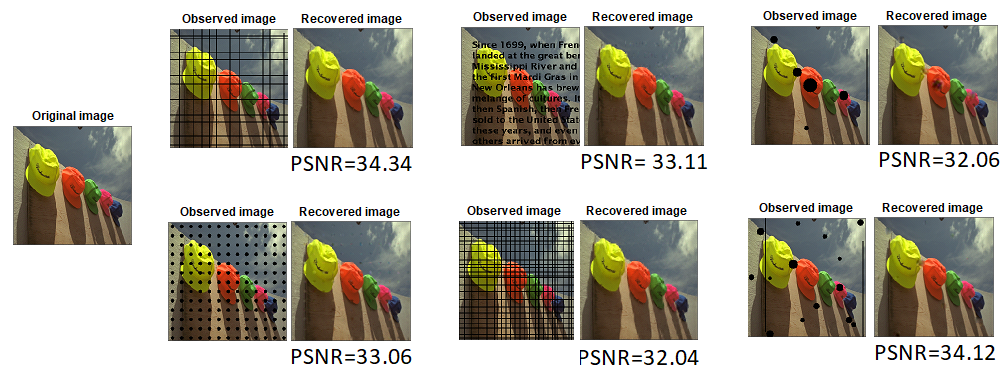}
    \caption{\small{The PSNR achieved by Algorithm \ref{ALg_2}} for example \ref{EX:3} in the structured missing pattern case (six types of missing patterns).}
    \label{figstru_miss:ex3}
\end{figure}

\begin{figure}
\begin{center}
\includegraphics[width=1.05\linewidth]{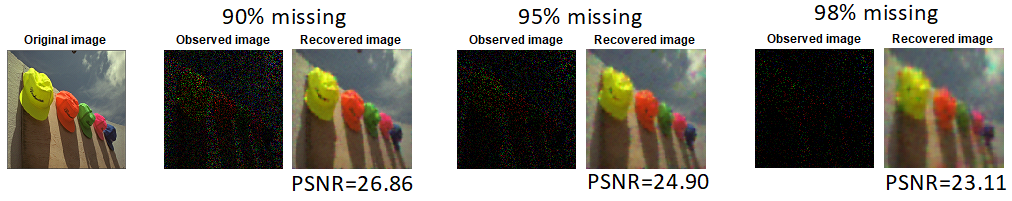}
\caption{\small{The PSNR  for Algorithms \ref{ALg_2}} for images with high missing ratios, for example \ref{EX:3}.}\label{fig:extreme_case}
\end{center}
\end{figure}


\begin{figure}
\begin{center}
\includegraphics[width=0.75\linewidth]{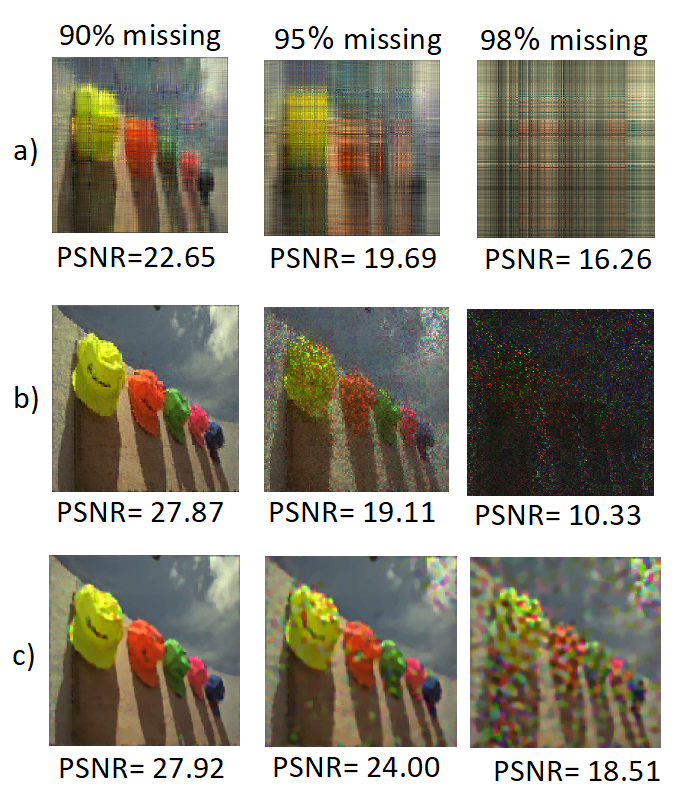}
\caption{\small{The completion results obtained by a) Bayesian tensor CPD (Canonical Polyadic decomposition) completion \cite{zhao2015bayesian}, b) tensor ring completion \cite{huang2020provable}, and c) tensor train completion  \cite{ding2019low} for images with high missing ratios, for example \ref{EX:3}.}}\label{fig:comp}
\end{center}
\end{figure}
  
\end{exa}

\begin{exa}\label{ex_4}
For this experiment, we use the Kodak dataset\footnote{https://r0k.us/graphics/kodak/}, comprising 23 images of size either $768 \times 512 \times 3$ or $512 \times 768 \times 3$. For consistency, all images were resized to $256 \times 256 \times 3$. We remove 80\% of pixels in all images and use tubal tensor rank $R=40$, the oversampling parameter, $P=10$, and the power iteration parameter $q=2$. We use 100 iterations of the completion algorithm for the image recovery process. The PSNRs of the recovered images for the whole Kodak dataset are shown in Figure \ref{fig:ex4}. This experiment demonstrated the feasibility of the completion algorithm, which is empowered by the proposed randomized block Krylov method. On average, completion of each image took 13 seconds; however, the algorithm's computation can be significantly accelerated on GPU devices. Furthermore, when processing a batch of images, the completion process can be executed in parallel for each image, substantially reducing the total processing time.

\begin{figure}
    \centering
    \includegraphics[width=1.15\linewidth]{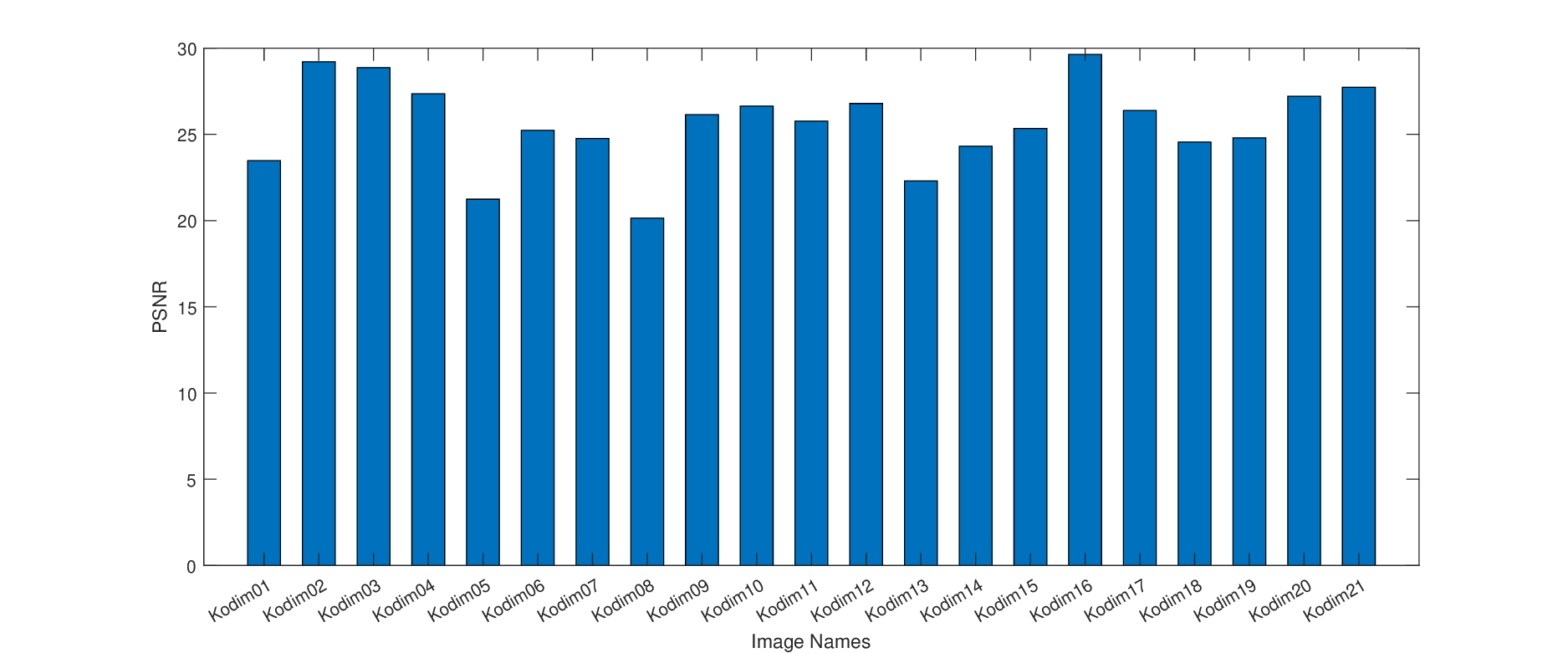}
    \caption{The PSNR of the recovered images for all images of the Kodak dataset for Example \ref{ex_4}.}
    \label{fig:ex4}
\end{figure}
\end{exa}

\section{Conclusion and future works}\label{sec:conclu}
 In this paper, we introduced the randomized block Krylov subspace method for approximating the truncated tensor singular value decomposition. We have currently examined this problem from a theoretical perspective and conducted numerical experiments using both synthetic and real-world datasets to validate our theoretical findings. Our numerical results clearly showed that the proposed algorithm is not only fast but also capable of delivering accurate approximations of the truncated T-SVD efficiently. We also present two applications of this randomized algorithm in the fields of data compression and completion. Currently, we are exploring the use of the randomized block Krylov subspace method for other types of tensor decompositions. A promising direction for future work is the application of our randomized block Krylov method to enhance the adversarial robustness of deep neural networks, a topic we are currently exploring.

 \section{Acknowledgment}
The work was supported by the Ministry of Economic Development of the Russian Federation under Agreement No. 139-10-2025-034 dd. 19.06.2025, IGK 000000C313925P4D0002.

 \bibliographystyle{elsarticle-num} 
 \bibliography{cas-refs}

\end{document}